\documentclass[12pt]{amsart}
\usepackage{amsfonts, amssymb, mathscinet}
\usepackage{amsaddr}
\usepackage[margin=1.25in]{geometry}
\usepackage{array}
\usepackage[all]{xypic}
\usepackage{ytableau}

\definecolor{untgreen}{RGB}{5,144,51} 
\usepackage[breaklinks=true, 
colorlinks, allcolors=untgreen, linktocpage ]{hyperref}

\numberwithin{equation}{subsection} 
\usepackage[capitalise]{cleveref} 
\crefname{equation}{}{} 
\crefname{subsection}{}{} 
\crefname{section}{\S\kern -.5 ex}{\S\S\kern -1 ex}

\theoremstyle{plain}
\newtheorem{lemma}[subsection]{Lemma}
\newtheorem{proposition}[subsection]{Proposition}

\newtheorem{theorem}[subsection]{Theorem} 
\newtheorem*{theorem*}{Theorem}
\newtheorem{corollary}[subsection]{Corollary} 

\theoremstyle{definition}%

\newcommand\CC{{\mathcal C}}
\newcommand\CD{{\mathcal D}}

\newcommand\CH{{\mathcal H}}
\newcommand\CI{{\mathcal I}}

\newcommand\CO{{\mathcal O}}
\newcommand\CP{{\mathcal P}}

\newcommand\CS{{\mathcal S}}
\newcommand\CT{{\mathcal T}}

\newcommand\BBC{{\mathbb C}}

\newcommand\BBF{{\mathbb F}}

\newcommand\BBZ{{\mathbb Z}}

\newcommand\ba{\mathbf a}

\newcommand\bq{\mathbf q}

\newcommand\ibar{\overline{\imath}}
\newcommand\jbar{\overline{\jmath}}

\newcommand\mbar{\overline{m}}

\newcommand\Shat{\widehat{S}}

\newcommand\Atilde{\widetilde{A}}

\newcommand\Gtilde{\widetilde{G}}

\newcommand\Rtilde{\widetilde R}
\newcommand\Stilde{\widetilde{S}}
\newcommand\vtilde{\tilde{v}}

\DeclareMathOperator{\GL}{GL}
\DeclareMathOperator{\Hom}{Hom} 
\DeclareMathOperator{\Ind}{Ind}
\newcommand\id{\operatorname{id}} 

\newcommand\inverse{^{-1}}
\renewcommand\th{{^{\text{th}}}}

\newcommand\compnb[2]{\CC(#1,#2)}

\newcommand\bpart[2]{#1\text{-}\kern-.1em\CP(#2)}

\newcommand\CHt{\widetilde{\mathcal H}}

\newcommand\bv{\mathbf v}
\renewcommand\mod[1]{#1\textrm{-mod}}
\newcommand\SYT{{\mathrm SYT}}
\newcommand\YT{{\mathrm YT}}

\newcommand\Vbar{\overline{V}}

\allowdisplaybreaks

\begin{document}

\title[Yokonuma-type Hecke algebras]{Modules for Yokonuma-type Hecke
  algebras}
\author{Ojas Dav\'e}
\address{Department of Mathematics and Statistics\\
Radford University\\ 
Radford, VA 24142, USA}
\email{odave@radford.edu}

\author{J. Matthew Douglass}
\address{Department of Mathematics\\
  University of North Texas\\
  Denton, TX 76203, USA} 
\email{douglass@unt.edu}
\thanks{This work was partially supported by a grant from the Simons
  Foundation (Grant \#245399 to J.M.~Douglass). J.M.~Douglass would like to
  acknowledge that some of this material is based upon work supported by
  (while serving at) the National Science Foundation.}


\begin{abstract}\noindent
  This paper describes the module categories for a family of generic Hecke
  algebras, called Yokonuma-type Hecke algebras. Yokonuma-type Hecke
  algebras specialize both to the group algebras of the complex reflection
  groups G(r,1,n) and to the convolution algebras of (B',B')-double cosets
  in the group algebras of finite general linear groups, for certain
  subgroups B' consisting of upper triangular matrices. In particular,
  complete sets of inequivalent, irreducible modules for semisimple
  specializations of Yokonuma-type Hecke algebras are constructed.
\end{abstract}

\maketitle


\section{Introduction}\label{sec:intro}

\subsection{}
Suppose $n$ is a positive integer and $q$ is a prime power. Let
$G=\GL_n(\BBF_q)$ be the group of all invertible $n\times n$ matrices with
entries in the finite field $\BBF_q$ and let $U$ denote the subgroup of $G$
consisting of all unipotent upper triangular matrices. The natural action of
$G$ on the set of cosets $G/U$ determines the permutation representation
$\Ind_U^G (1_U)$ of $G$ (over $\BBC$). The endomorphism algebra of this
representation is anti-isomorphic to the subalgebra $\CH_{\BBC}(G,U)$ of the
group algebra $\BBC[G]$ consisting of functions that are constant on
$(U,U)$-double cosets. For a finite Chevalley group $G'$ with maximal
unipotent group $U'$, Yokonuma \cite{yokonuma:structure} gave a presentation
of $\CH_{\BBC}(G',U')$ and described some of its structure. When the prime
power $q$ is replaced by an indeterminate $\bq$, Yokonuma's presentation
(extended to $G$, see \cite{juyumaya:representation}) defines an algebra
over the polynomial ring $\BBZ[\bq]$. We call these algebras one-parameter
Yokonuma-Hecke algebras. The irreducible $\CH_{\BBC}(G,U)$-modules have been
constructed and classified by Thiem \cite{thiem:thesis} using a weight
space-type decomposition. Using combinatorial methods, Chlouveraki and
Poulain d'Andecy \cite{chlouverakipoulaindandecy:representation} have
constructed and classified the irreducible modules for one-parameter
Yokonuma-Hecke algebras.

Now suppose that $a$ and $b$ are relatively prime positive integers such
that $q-1=ab$. Then the multiplicative group of $\BBF_q$ (a finite cyclic
group of order $q-1$) has a unique factorization as $F_a F_b$, where
$|F_a|=a$ and $|F_b|=b$. Let $H_a$ (resp.~$H_b$) denote the subgroup of $G$
consisting of diagonal matrices with entries in $F_a$ (resp.~$F_b$). For
example, if $l$ is a prime that does not divide either $q$ or $n!$ and $a$
is the $l$-part of $q-1$, then $H_a$ is a Sylow $l$-subgroup of $G$. Set
$B_a=H_aU$. Notice that if $a=1$, then $B_a=U$ and if $a=q-1$, then $B_a=B$
is the Borel subgroup of upper triangular matrices in $G$. Yokonuma's
results in \cite{yokonuma:structure} can be used to describe the algebra
$\CH_{\BBC}(G,B_a)$ as a subalgebra of $\CH_{\BBC}(G,U)$. A presentation of
the algebra $\CH_{\BBC}(G,B_a)$ was given in
\cite[\S2]{alhaddaddouglass:generic}. When the prime power $q$ and the divisor
$a$ are replaced by indeterminates $\bq$ and $\ba$, the presentation in
\cite[\S3]{alhaddaddouglass:generic} defines an algebra over the polynomial ring
$\BBZ[\bq,\ba]$. We call these algebras Yokonuma-type Hecke algebras or
two-variable Yokonuma-Hecke algebras.

The Iwahori-Hecke algebra of $G$ is a $\BBZ[\bq]$-algebra that specializes
both to the centralizer algebra $\CH_{\BBC}(G,B)$ and to the group algebra
$\BBC [W]$, where $W$ is the group of permutation matrices in
$G$. Similarly, Yokonuma-type Hecke algebras specialize both to the
centralizer algebra $\CH_{\BBC}(G,B_a)$ and to the group algebra
$\BBC [WH_b]$. The group $WH_b$ is abstractly isomorphic to the complex
reflection group denoted by $G(b,1,n)$.

The main results in this paper are a ``block'' decomposition of the module
category of a Yokonuma-type Hecke algebra in \cref{sec:block}, and an
explicit construction and classification of the simple modules for
semisimple specializations of these algebras in \cref{sec:module}. Our
approach is based on Thiem's, but even for the algebras considered in
\cite{thiem:thesis} we give a more detailed description of the structure of
these module categories and finer information about the structure of the
simple modules.

Yokonuma-type Hecke algebras have been studied in several papers using
different presentations. To help the reader, we explain in \cref{sec:last}
the precise relationships between the presentation used by Chlouveraki and
Poulain d'Andecy \cite[\S2.1]{chlouverakipoulaindandecy:representation}, the
presentation used by Jacon and Poulain d'Andecy
\cite[\S2.3]{jaconpoulaindandecy:isomorphism}, and the presentation in this
paper. The relationship between Yokonuma's original presentation of
$\CH_{\BBC}(G,U)$ and the presentation found by Juyumaya
\cite[\S2]{juyumaya:nouveaux} is also sketched.

The results in this paper both supplement and complement results about the
structure of $\CH$ and $\CH$-modules given in
\cite{chlouverakipoulaindandecy:representation} and
\cite{jaconpoulaindandecy:isomorphism}. In particular, the constructions in
\cref{sec:block} and \cref{sec:module} lead to a uniform approach to the
different constructions of simple modules given in these two papers.

Let $\CH$ be the Yokonuma-type Hecke algebra defined in \cref{ssec:Hbn}
below, with scalars extended to $\BBC[\bq, \bq\inverse, \ba]$. As described
in \cref{sec:last}, $\CH$ is the algebra $Y_{d,n}$ defined in
\cite[\S2.3]{jaconpoulaindandecy:isomorphism} and the
$\BBC[\bq, \bq \inverse]$-algebra $Y_{d,n}(\bq)$ in
\cite[\S2.1]{chlouverakipoulaindandecy:representation} is a specialization
of $\CH$. With the notation introduced in the next section, there is a
direct sum decomposition
\begin{equation}
  \label{eq:dec1}
  \CH \cong \bigoplus_{\alpha\in \CC} \CH e_\alpha,  
\end{equation}
where $\CH e_\alpha$ is a two-sided ideal in $\CH$ and the index $\alpha$
determines a Young subgroup $\Sigma_\alpha$ of the symmetric group
$\Sigma_n$. Lusztig \cite[\S34]{lusztig:disconnectedVII} and Jacon and
Poulain d'Andecy \cite[\S3]{jaconpoulaindandecy:isomorphism} construct
explicit isomorphisms between $\CH e_\alpha$ and the matrix ring $M_{r_\alpha}
(\CI^\alpha)$, where $\CI^\alpha$ is the Iwahori-Hecke algebra of
$\Sigma_\alpha$. The block decomposition given in \cref{sec:block} may be
viewed as a categorical framework underlying these isomorphisms.

The simple $\BBC(\bq)Y_{d,n}(\bq)$-modules have been constructed by
Chlouveraki and Poulain d'Andecy
\cite[\S5]{chlouverakipoulaindandecy:representation} by defining an action
of the generators of $\BBC(\bq)Y_{d,n}(\bq)$ on a basis and checking
directly that the defining relations of $\BBC(\bq)Y_{d,n}(\bq)$ hold.
Somewhat more directly, as observed in
\cite[\S4.1]{jaconpoulaindandecy:isomorphism}, it follows from the
decomposition \eqref{eq:dec1} and the isomorphisms $\CH e_\alpha \cong
M_{r_\alpha}(\CI^\alpha)$ that every simple $\CH$-module is the space of
column vectors of size $r_\alpha$ with entries in $M$, where $M$ is a simple
$\CI^\alpha$-module.

In analogy with the so-called method of little groups
\cite[11.8]{curtisreiner:methodsI} (which applies to the group $G(b,1,n)$)
it follows from block decomposition in \cref{sec:block} that every simple
$\CH$-module is of the form $\CH\otimes_{\CH_\alpha} M$, where $\CH_\alpha$
is a subalgebra of $\CH$ that is isomorphic to $\CI^\alpha$ and $M$ is a
simple $\CH_\alpha$-module. As an $\CH_\alpha$-module, $\CH$ is free with
rank $r_\alpha$. In this way we obtain a coordinate-free construction of the
simple modules in \cite{jaconpoulaindandecy:isomorphism}. There are several
constructions of $\CI^\alpha$-modules that yield a complete set of simple
modules for any semisimple specialization of $\CI^\alpha$. Using Hoefsmit's
construction \cite{hoefsmit:representations} and the induction result from
\cref{sec:block}, in \cref{sec:module} a family of $\CH$-modules is defined
that specializes to the modules for one-parameter Yokonuma-Hecke algebras
constructed in \cite[\S5]{chlouverakipoulaindandecy:representation}. The
construction in \cref{sec:module} is valid more generally for Yokonuma-type
Hecke algebras and gives more information about these modules than is
immediately evident from their description in
\cite[\S5]{chlouverakipoulaindandecy:representation}. A new feature of the
construction in this paper is that it highlights how the natural action of
$W$ on the set of $b$-tableau is reflected in the module structure of these
$\CH$-modules, see \cref{pro:factab} and \cref{thm:irrep}.

\subsection{Notation}
For a positive integer $l$, $[l]$ denotes the set $\{1, \dots, l\}$ and
$\Sigma_l$ denotes the group of permutations of $[l]$.

\section{A block decomposition of
  \texorpdfstring{$\mathbf{\CH}$}{H}}\label{sec:block}

In this section we begin with the definition of the Yokonuma-type Hecke
algebra $\CH$ from \cite{alhaddaddouglass:generic}, then we describe the
algebra structure of $\CH$ and give some specific specializations. The main
result is a block decomposition of the category of $\CH$-modules.

\subsection{}\label{ssec:Hbn}
Suppose from now on that $n$ and $b$ are positive integers and that $\bq$
and $\ba$ are indeterminates. Define $\CH_{b,n}$ to be the unital
$\BBZ[\bq, \ba]$-algebra with generators
\[
\text{$T_1$, \dots, $T_n$, $R_1$, \dots $R_{n-1}$}
\]
and relations
\begin{alignat}{2}
  &T_j^{b}=1 \qquad&& \text{for $j\in [n]$,} \tag{$\textrm
    r_1$} \label{rel:1}  \\
  &T_jT_{j'}= T_{j'}T_j \qquad&& \text{for $j, j'\in[n]$,}
  \tag{$\textrm r_2$} \label{rel:2} \\
  &T_jR_{i}=R_{i}T_{s_i(j)} \qquad&& \text{for $j\in [n]$ and
    $i\in [n-1]$,} \tag{$\textrm r_3$} \label{rel:3} \\
  &R_{i}R_{i'}=R_{i'}R_{i} \qquad&& \text{for $i, i'\in [n-1]$ with
    $|i-i'|>1$,} \tag{$\textrm r_4$} \label{rel:4}\\
  &R_{i}R_{i+1}R_{i}= R_{i+1}R_{i}R_{i+1} \qquad&& \text{for $i\in [n-2]$,
    and} \tag{$\textrm r_5$} \label{rel:5}\\
  &R_{i}^2=\bq +\ba T_{i}^{(b^2-b)/2} E_i R_i \qquad&& \text{for $i\in
    [n-1]$,} \tag{$\textrm r_6$} \label{rel:6}
\end{alignat}
where $s_i = (i\ i+1)$ is the transposition in the symmetric group
$\Sigma_n$ that switches $i$ and $i+1$ and
\[
E_i= \sum_{k=0}^{b-1} T_i^k T_{i+1}^{-k}.
\]
If $b$ is odd, then $T_{i}^{(b^2-b)/2}=1$, and if $b$ is even, then
$T_{i}^{(b^2-b)/2}= T_{i}^{b/2}$ is a square root of $1$. Notice that
$\CH_{1,n}$ is the Iwahori-Hecke algebra of type $A_{n-1}$ with parameters
$\bq$ and $\ba$.

\subsection{}
Let $\mu_b=\langle \zeta \rangle$ be the group of $b\th$ roots of unity in
$\BBC$ and let $D$ denote the group of diagonal matrices in $\GL_n(\BBC)$
with diagonal entries in $\mu_b$. Suppose $d\in D$ and that the
$(j,j)$-entry of $d$ is $\zeta^{p_j}$ for $j\in [n]$. Define
\[
t_d=T_1^{p_1} \dotsm T_n^{p_n}\in \CH_{b,n}.
\]
Then if $\CD$ denotes the subalgebra of $\CH_{b,n}$ generated by $\{T_1,
\dots, T_n\}$, it follows from relations \cref{rel:1} and \cref{rel:2} that
the rule $d\mapsto t_d$ defines a $\BBZ[\bq, \ba]$-algebra isomorphism
between the group algebra $\BBZ[\bq, \ba][D]$ and $\CD$.

\subsection{}
Let $W$ denote the group of permutation matrices in $\GL_n(\BBC)$. The
matrices in $W$ permute the standard basis vectors $v_1$, \dots, $v_n$ of
$\BBC^n$, where $v_i$ is the column vector in $\BBC^n$ whose only non-zero
entry is a 1 in the $i\th$ coordinate. We use the bijection $v_i
\leftrightarrow i$ between $\{v_1, \dots, v_n\}$ and $\{1, \dots, n\}$ to
identify permutation matrices with permutations. Precisely, for $w\in W$,
$w$ also denotes the permutation in the symmetric group $\Sigma_n$ defined
by
\[
wv_i=v_{w(i)}.
\]
Whether $w\in W$ denotes a matrix or a permutation will always be clear from
context.

For $i\in [n-1]$ let $s_i$ denote the permutation matrix in $W$ obtained
from the identity matrix by interchanging rows $i$ and $i+1$ and set
$\CS=\{s_1, \dots, s_{n-1}\}$. (Note that we considered $s_i$ as a
permutation in relation \cref{rel:3}.) Then $\CS$ is a set of Coxeter
generators for $W$ and so determines a length function on $W$ and the notion
of a reduced expression of an element of $W$. Suppose $w\in W$ and $s_{i_1}
\dotsm s_{i_k}$ is a reduced expression for $w$. Define
\[
t_w= R_{i_1} \dotsm R_{i_k}\in \CH_{b,n}.
\]
It follows from relations \cref{rel:4}, \cref{rel:5}, and Matsumoto's
Theorem that $t_w$ is well-defined.

\subsection{}
The group $W$ normalizes $D$ and the semidirect product $WD$ is the complex
reflection group denoted by $G(b,1,n)$. For $x=wd\in WD$, define
\[
t_{x}= t_wt_d\in \CH_{b,n}.
\]
The element $t_x$ is well-defined because $W\cap D=1$. Extend the length
function $\ell$ on $W$ to all of $WD$ by defining $\ell(x)= \ell(w)$ when
$x=wd$. The next theorem is proved in \cite[\S3]{alhaddaddouglass:generic}.

\begin{theorem}\label{thm:rel}
  The set $\{\, t_x\mid x\in WD\,\}$ is an $\BBZ[\bq, \ba]$-basis of
  $\CH_{b,n}$. For $x$ in $WD$, $d$ in $D$, and $s=s_i$ in $\CS$ the
  following relations hold:
  \begin{align*}
    t_xt_d&= t_{xd} \\
    t_dt_x&=t_{dx} \\
    t_st_x&= \begin{cases} t_{sx} &\text{if $\ell(sx)>\ell(x)$}\\
      \bq t_{sx}+ \ba T_{i}^{(b^2-b)/2} E_i t_x &\text{if
        $\ell(sx)<\ell(x)$} \end{cases} \\
    t_xt_s&= \begin{cases}
      t_{xs} &\text{if $\ell(xs)>\ell(x)$}\\
      \bq t_{xs}+ \ba t_x T_{i}^{(b^2-b)/2} E_i &\text{if
        $\ell(xs)<\ell(x)$.}
  \end{cases}
  \end{align*}
\end{theorem}

\subsection{Specialization}\label{ssec:sp}
A specialization of a commutative ring with identity $R$ is a ring
homomorphism $f\colon R\to k$, where $k$ is a commutative ring with identity
and $f(1)=1$.

Suppose $H$ is an $R$-algebra and $f\colon R\to k$ is a specialization. Then
$f$ determines an $R$-module structure on $k$ and we may form the
specialized algebra
\[
{}_fH=k\otimes_{R} H.
\] 
If $V$ is an $H$-module, define ${}_fV= k\otimes_{R} V$. Then ${}_fV$ is
naturally an ${}_fH$-module and specialization determines a functor from the
category of $H$-modules to the category of ${}_fH$-modules. In the special
case when $R\subseteq k$ and $\imath \colon R\to k$ is the inclusion, we
replace the subscript $\imath$ by $k$, so ${}_k H= k\otimes_{R}H$ and ${}_kV
=k\otimes_{R} V$. In addition, for $h\in H$ and $v\in V$, we sometimes abuse
notation slightly and denote $1\otimes h$ by $h$ and $1\otimes v$ by
$v$. With this convention, ${}_kH$ is the space of all $k$-linear
combinations of elements in $H$, and similarly for ${}_kV$.

There are two types of specializations of $\BBZ[\bq, \ba]$ that relate
$\CH_{b,n}$ to the representation theory of finite groups.
\begin{enumerate}
\item If $f\colon \BBZ[\bq, \ba]\to k$ is a specialization such that
  $f(\bq)=1$ and $f(\ba)=0$, then clearly
  \[ 
  {}_f \CH_{b,n}\cong k[WD] \cong k[G(b,1,n)].
  \]
\item If $f\colon \BBZ[\bq, \ba]\to k$ is a specialization such that $k$
  is a field with characteristic zero, $f(\bq)=q$ is a prime power, and
  $f(\ba)=a$, where $q-1=ab$ with $a$ and $b$ relatively prime, then with
  the notation used in \cref{sec:intro}, it is shown in
  \cite{alhaddaddouglass:generic} that
  \[ 
  {}_f \CH_{b,n}\cong \CH_k(G,B_a),
  \]
  where the algebra on the right-hand side is the convolution algebra of
  $k$-valued functions on $G=\GL_n(\BBF_q)$ that are constant on $(B_a,
  B_a)$-double cosets. In this isomorphism $WD$ is identified with a group
  of monomial matrices in $G$ and the basis element $1\otimes t_{x}$ of
  ${}_f \CH_{b,n}$ corresponds to $|B_a|\inverse$ times the characteristic
  function of the double coset $B_axB_a$.

  In particular, if $b=q-1$ and $a=1$, then ${}_f \CH_{q-1,n}\cong
  \CH_k(G,U)$ and the presentation given by relations
  \cref{rel:1}--\cref{rel:6} is essentially that given by Yokonuma
  \cite{yokonuma:structure}.
\end{enumerate}

Taking $k=\BBC$ in both cases above, it follows from the results in
\cite{alhaddaddouglass:generic} and Tits' Deformation Theorem
\cite[68.17]{curtisreiner:methodsII} that $\CH_{\BBC}(G,B_a) \cong
\BBC[G(b,1,n)]$, so as a special case, $\CH_{\BBC}(G,U) \cong
\BBC[G(q-1,1,n)]$.

\subsection{The ring \texorpdfstring{$\mathbf{A}$}{A} and the algebra
  \texorpdfstring{$\mathbf{\CH}$}{H}} \label{ssec:AH} 

Recall that $n$ and $b$ are positive integers, $\bq$ and $\ba$ are
indeterminates, and $\zeta$ is a complex primitive $b\th$-root of unity. In
order to continue we need to replace the scalar ring $\BBZ[\bq, \ba]$ by the
smallest sufficiently large extension in which the constructions in this
section and the next can be carried out.

Define $A$ to be the smallest subring of an algebraic closure of $\BBC(\bq,
\ba)$ with the following properties.
\begin{enumerate}
\item $A$ contains $1$, $\bq$, $\ba$, $\zeta$, and a square root of
  $b^2\ba^2+4\bq$, denoted by $\sqrt{b^2 \ba^2+4\bq}$.
\item $2$, $b$, and $\bq$ are units in $A$.
\item For $k\in [n-2]$ the elements $1+x +\dotsm+ x^k$ are units in $A$,
  where
  \[
  x= \frac {b^2\ba^2+ b\, \ba \sqrt{b^2 \ba^2+4\bq}+2\bq} {2\bq}.
  \]
  (It is shown in \cref{lem:unit} that $x$ is not a root of unity, so $1+x
  +\dotsm+ x^k\ne 0$ for all $k>0$.)
\end{enumerate}
More concisely, using the standard notation for localization and adjoining
elements to a subring,
\[
A= X\inverse\left( Q\inverse \left(Y\inverse \BBZ[\bq, \ba] \right)
  [\sqrt{b^2 \ba^2+4\bq}] \right),
\]
where $X= \{\, \sum_{i=0}^{k}x^i \mid k\in [n-2]\,\}$, $Q= \{\, \bq^i\mid
i\geq0\,\}$, and $Y= \{\, 2^ib^j\mid i,j\geq0\,\}$. By construction $A$ is
an integral domain. Let $K$ denote the quotient field of $A$.

Define 
\[
\CH= A\otimes_{\BBZ[\bq, \ba]} \CH_{b,n}= {}_A\CH_{b,n}.
\]
In the rest of this paper we will be concerned with the $A$-algebra $\CH$
and its specializations.

\begin{lemma}\label{lem:unit}
  The element $x\in A$ is not a root of unity.
\end{lemma}

\begin{proof}
  Let $I$ be the ideal in $A$ generated by $b\ba -\bq+1$. Then $b\ba +I=
  \bq-1+I$ and so $b^2\ba^2+4\bq +I= (\bq+1)^2+I$. It follows that
  $\sqrt{b^2\ba^2+4\bq}+I= \pm(\bq+1)+I$. Therefore,
  \[
  x+I= \frac {b^2\ba^2+ b\ba \sqrt{b^2\ba^2+4\bq}+2\bq} {2\bq}+I =
  \frac{\bq^2+1 \pm (\bq^2-1)} {2\bq} +I = \bq^{\pm 1}+I.
  \]
  Now if $x^k=1$, then $\bq^{\pm k} -1\in I \cap \BBZ[\bq, \bq\inverse] =
  0$, and so $k=0$. Thus, $x$ is not a root of unity.
\end{proof}

\subsection{Characters, idempotents, and weights}\label{ssec:ciw}
In this section we give the straightforward adaption of Thiem's
constructions \cite[Chapter 6]{thiem:thesis} to the algebra $\CH$.

Let $X(D)$ denote the group of $A$-characters of $D$. That is, the group
of all group homomorphisms $\chi\colon D\to A^\times$, where $A^\times$
denotes the unit group of $A$. Note that $|X(D)|=b^n$ because $A$ contains
a primitive $b\th$ root of unity. An explicit construction of the characters
in $X(D)$ is given in \cref{ssec:oreps}.

For $\chi\in X(D)$ define
\[
e_\chi= b^{-n} \sum_{d\in D} \chi(d\inverse) t_d \in \CD.
\]
Then $e_\chi$ is the centrally primitive idempotent in $\CD$ with
$t_d e_\chi= \chi(d) e_\chi$ for $d\in D$. Moreover,
$\{\, e_\chi\mid \chi\in X(D)\,\}$ is an $A$-basis of $D$ and a complete set
of orthogonal idempotents in $\CD$, and hence in $\CH$. 

The group $W$ normalizes $D$ and so acts on $X(D)$ with $(w \cdot \chi)(d)=
\chi(w\inverse dw)$ for $w\in W$, $\chi\in X(D)$, and $d\in D$. Let $\CC$
denote the set of orbits of $W$ in $X(D)$, and for each orbit $\alpha\in \CC$
fix a representative $\chi_\alpha\in \alpha$.

Using \cref{thm:rel} it is straightforward to check that $\{\, t_we_\chi\mid
w\in W, \chi\in X(D)\,\}$ is an $A$-basis of $\CH$ and that
\begin{equation}
  \label{eq:we1}
  t_w e_\chi= e_{w\cdot \chi} t_w  
\end{equation}
for $w\in W$ and $\chi\in X(D)$.
 
Suppose $\alpha\in \CC$. Define
\[
\CH_\alpha= e_{\chi_\alpha} \CH e_{\chi_\alpha}\quad\text{and}\quad
e_\alpha= \sum_{\chi\in \alpha} e_\chi.
\]
Then $\CH_\alpha$ is a subalgebra of $\CH$ and $e_\alpha$ is an idempotent
in $\CD$.  

It is shown in \cite{alhaddaddouglass:generic} that $t_w$ is a unit in $\CH$
for each $w$ in $W$. Thus, by \cref{eq:we1},
\[
\CH e_\chi= \CH t_w e_\chi = \CH e_{w\cdot \chi} t_w,
\]
and so right multiplication by $t_w$ defines an $\CH$-module isomorphism
$\CH e_\chi\cong \CH e_{w\cdot \chi}$. This implies that $e_\chi\CH e_\chi
\cong e_{w\cdot \chi}\CH e_{w\cdot \chi}$ for all $w\in W$. In particular,
up to isomorphism, $\CH_\alpha$ depends only on $\alpha$ and not the choice
of $\chi_\alpha$.

It follows from \cref{eq:we1} that $e_\alpha$ is in the center of
$\CH$. Therefore, the set $\{\, e_\alpha \mid \alpha\in \CC\,\}$ is a
complete set of orthogonal, central idempotents in $\CH$, and there are
direct sum decompositions
\begin{equation*}
  \label{eq:tsi}
  \CH e_\alpha\cong \bigoplus_{\chi\in \alpha} \CH e_\chi \quad\text{and}\quad
  \CH\cong \bigoplus_{\alpha\in \CC} \CH e_\alpha \cong \bigoplus_{\alpha\in
    \CC} \bigg( \bigoplus_{\chi\in \alpha} \CH e_\chi \bigg),  
\end{equation*}
where each $\CH e_\alpha$ is a two-sided ideal.  

Now suppose $V$ is an $\CH$-module. For $\chi\in X(D)$ define
\[
V_\chi= e_\chi V = \{\, v\in V\mid t_dv=\chi(d) v\ \forall\ d\in D\,\}.
\]
Then $V_\chi$ is an $A$-submodule of $V$ and
\[
V\cong \bigoplus_{\chi\in X(D)} V_\chi
\]
is a decomposition of $V$ as a direct sum of $A$-modules. We can think of
characters in $X(D)$ as weights and then $V_\chi$ is the $\chi$-weight space
of $V$.

If $w\in W$, $d\in D$, and $v\in V_\chi$, then by \cref{thm:rel},
\begin{equation} 
  \label{eq:we2}
  t_d (t_w v)= t_w t_{w\inverse dw}v = \chi(w\inverse dw) t_w v= (w\cdot
  \chi)(d)\, t_w v  ,
\end{equation}
and so multiplication by $t_w$ defines an $A$-linear isomorphism
$V_\chi\cong V_{w\cdot \chi}$. It follows that
\[
\sum_{w\in W} t_w V_\chi =\sum_{w\in W} V_{w\cdot \chi}
\]
is an $\CH$-submodule of $V$.

For $\alpha\in \CC$ define
\[
V_\alpha= e_\alpha V=\sum_{\chi\in \alpha} V_\chi.
\]
Then $V_\alpha$ is an $\CH$-submodule of $V$ and
\begin{equation}
  \label{eq:vd}
  V\cong \bigoplus_{\alpha\in \CC} V_\alpha
\end{equation}
is a decomposition of $V$ as a direct sum of $\CH$-submodules.

\subsection{Blocks}\label{ssec:bk}
For a ring $R$, let $\mod R$ denote the category of left $R$-modules.

Suppose $\alpha\in \CC$. Define $\CO_\alpha$ to be the full subcategory of
$\mod \CH$ with objects the collection of all $\CH$-modules $V$ such that
$V=V_\alpha$. At the risk of abusing terminology we say that $\CO_\alpha$ is
a block of $\CH$.

If $V$ and $V'$ are $\CH$-modules and $\psi\colon V\to V'$ is an
$\CH$-module homomorphism, then $\psi(V_\alpha)\subseteq V'_\alpha$ for all
$\alpha\in \CC$. Thus, if $\alpha, \beta\in \CC$ with $\alpha\ne \beta$, $V$
is a module in $\CO_\alpha$, and $V'$ is a module in $\CO_\beta$, then
$\Hom_{\CH}(V,V')=0$. Hence, it follows from the decompositions in
\cref{ssec:ciw}\cref{eq:vd} that
\begin{equation}\label{eq:bk}
  \mod{\CH} \simeq \bigoplus_{\alpha\in \CC} \CO_\alpha   
\end{equation}
is a decomposition of $\mod{\CH}$ as a direct sum of abelian subcategories.

If $V$ is an $\CH$-module, then $V_{\chi_\alpha} =e_{\chi_\alpha}V$ is an
$\CH_\alpha$-module. Let
\[
G_\alpha\colon \CO_\alpha \to \mod {\CH_\alpha}
\]
be the restriction functor defined on objects by $G_\alpha(V)=
V_{\chi_\alpha}$ and on homomorphisms by restriction.

If $V'$ is an $\CH_\alpha$-module, then $e_{\chi_\alpha} \big(\CH
e_{\chi_\alpha} \otimes_{\CH_\alpha} V' \big)= \CH_\alpha
\otimes_{\CH_\alpha} V' \cong V'$ and so the $\CH$-module $\CH
e_{\chi_\alpha} \otimes_{\CH_\alpha} V'$ is in $\CO_\alpha$.  Let
\[
F_\alpha \colon \mod {\CH_\alpha}\to \CO_\alpha
\]
be the induction functor defined on objects by $F_\alpha (V')= \CH
\otimes_{\CH_\alpha} V' = \CH e_{\chi_\alpha} \otimes_{\CH_\alpha} V'$ and
on morphisms by $F_\alpha (\psi)= \id\otimes \,\psi$.

Suppose $f\colon A\to k$ is a specialization. Then $f(b)$ is a unit in $k$
and it follows that $f(\zeta)$ is a primitive $b\th$ root of unity. (If
$f(\zeta)$ is an $l\th$ root of unity and $b=lm$, then $1+f(\zeta)^l+ \dots
+ (f(\zeta)^l)^{m-1} = m $ is not equal to zero in $k$. But $\zeta^l$ is an
$m\th$ root of unity, so $1+f(\zeta^l)+ \dots + f(\zeta^l)^{m-1} = 0$, which
is a contradiction unless $m=1$ and $l=b$.) Therefore, the rule $\chi\mapsto
f\circ \chi$ defines an isomorphism between $X(D)$ and $X_k(D)$ so the
preceding constructions all apply in $\mod {{}_f\CH}$ and there are
categories ${}_f\CO$ and functors ${}_fF_\alpha$ and ${}_fG_\alpha$.

The next theorem is an analog of results of Lusztig
\cite[\S34]{lusztig:disconnectedVII} and Jacon and Poulain d'Andecy
\cite[\S3]{jaconpoulaindandecy:isomorphism}.

\begin{theorem}\label{thm:eqv}
  Suppose $\alpha\in \CC$ and $f\colon A\to k$ is a specialization. Then the
  pair of functors $({}_fF_\alpha ,{}_fG_\alpha)$ is an adjoint equivalence
  of categories. In particular, the block ${}_f\CO_\alpha$ of ${}_f\CH$ is
  naturally equivalent to the category of ${}_f\CH_\alpha$-modules.
\end{theorem}

\begin{proof}
  To help minimize subscripts, in this proof we suppress the specialization
  from the notation. For example we denote ${}_fF_\alpha$, ${}_f\CH$, and
  $1\otimes e_{\chi_\alpha}$, simply by $F_\alpha$, $\CH$, and
  $e_{\chi_\alpha}$, respectively.

  It is well-known and straightforward to check that $(F_\alpha, G_\alpha)$
  is an adjoint pair. As observed above, if $V'$ is an $\CH_\alpha$-module,
  then $G_\alpha(F_\alpha (V')) =e_{\chi_\alpha} \big(\CH e_{\chi_\alpha}
  \otimes_{\CH_\alpha} V' \big)$ is naturally isomorphic to $V'$ and so
  $G_\alpha F_\alpha $ is naturally equivalent to the identity functor. To
  complete the proof it is enough to show that if $V$ is an $\CH$-module in
  $\CO_\alpha$ then the natural map
  \[
  \gamma\colon \CH e_{\chi_\alpha}\otimes_{\CH_\alpha} V_{\chi_\alpha}\to
  V\quad\text{with}\quad \gamma(h\otimes v)=hv
  \]
  is an $\CH$-module isomorphism. The map $\gamma$ is obviously $\CH$-linear
  and so it is enough to show that it is an $A$-module isomorphism.

  Suppose $w\in W$, $\chi\in X(D)$ with $\chi \ne \chi_\alpha$, and
  $v\in V_{\chi_\alpha}$. Then
  \[
  t_we_\chi\otimes v= t_we_\chi\otimes e_{\chi_\alpha}v =t_we_\chi
  e_{\chi_\alpha} \otimes v=0.
  \]
  Because the set $\{\, t_we_\chi\mid w\in W,\ \chi\in X(D)\,\}$ is an
  $A$-basis of $\CH$ it follows that the set $\{\, t_w\otimes v\mid w\in W,\
  v\in V_{\chi_\alpha}\,\}$ spans $F_\alpha G_\alpha(V)$. Using
  \cref{thm:rel} we have
  \[
  t_d(t_w\otimes v)= t_w t_{w\inverse dw} \otimes v= t_w t_{w\inverse dw}
  e_{\chi_\alpha} \otimes v = (w\cdot \chi_\alpha)(d)\, (t_w \otimes v),
  \]
  and so
  $F_\alpha G_\alpha(V)_{w\cdot \chi_\alpha}= t_w\otimes V_{\chi_\alpha}$.
  Moreover, if $w\cdot \chi_\alpha=\chi_\alpha$, then
  $t_we_{\chi_\alpha}= e_{\chi_\alpha} t_w e_{\chi_\alpha}\in \CH_\alpha$ by
  \cref{ssec:ciw}\cref{eq:we1}, and so
  $t_w\otimes v= 1\otimes t_w e_{\chi_\alpha} v\in 1\otimes
  V_{\chi_\alpha}$.  It follows that
  \[
  F_\alpha G_\alpha(V)_{\chi_\alpha}= \sum_{w\cdot \chi_\alpha=\chi_\alpha}
  t_w\otimes V_{\chi_\alpha} = 1\otimes V_{\chi_\alpha}.
  \]
  The restriction of $\gamma$ to $1\otimes V_{\chi_\alpha}$ is obviously an
  isomorphism onto $V_{\chi_\alpha}$. Because $F_\alpha G_\alpha(V)$ and $V$
  are in $\CO_\alpha$ and left multiplication by the elements $t_w$ for
  $w\in W$ transitively permutes the non-zero weight spaces, the
  multiplication map $\gamma$ carries $F_\alpha G_\alpha(V)_\chi$
  isomorphically to $V_\chi$ for all $\chi\in X(D)$. This implies that
  $\gamma$ is an isomorphism.
\end{proof}

It is not hard to see that $\CO_\alpha$ is naturally isomorphic to the
category of $\CH e_\alpha$-modules and so it follows from the theorem that
$\CH e_\alpha$ and $\CH_\alpha$ are Morita equivalent. Jacon and Poulain
d'Andecy prove a more explicit result in
\cite{jaconpoulaindandecy:isomorphism}. Using the choice of orbit
representatives in \cref{ssec:oreps}, their arguments can be easily adapted
to construct an explicit isomorphism $\CH e_\alpha \cong
M_{r_\alpha}(\CH_\alpha)$, where $r_\alpha=|\alpha|$.

\section{Construction of \texorpdfstring{$\mathbf{\CH}$}
  {H}-modules}\label{sec:module}

The block decomposition in \cref{ssec:bk}\cref{eq:bk} and \cref{thm:eqv}
reduce the study of $\CH$-modules to the study of $\CH_\alpha$-modules. It
is shown below that $\CH_\alpha$ is isomorphic to a tensor product of
Iwahori-Hecke algebras of type $A$. Irreducible representations for
Iwahori-Hecke algebras of type $A$ have been constructed by various
authors. In this paper we use Hoefsmit's construction to define a family of
$\CH_\alpha$-modules. Inducing these modules to $\CH$ we obtain a family of
$\CH$-modules that gives rise to a complete set of irreducible
${}_f\CH$-modules, for any specialization $f\colon A\to k$ with the property
that ${}_f\CH_\alpha$ is semisimple for all $\alpha$. A particular case is
the inclusion of $A$ in its quotient field $K$. In this case, ${}_K\CH$ is a
split semisimple $K$ algebra and we obtain a complete set of irreducible
${}_K\CH$-modules, each of which is absolutely irreducible.

This section is organized as follows. First, the combinatorial constructions
we use are collected in \cref{ssec:oreps}--\cref{ssec:sum}. This is followed
in \cref{ssec:hm} by a review of Hoefsmit's construction in the form used
later. The structure of $\CH_\alpha$ is described and the
$\CH_\alpha$-modules we need are constructed in
\cref{ssec:Ha}--\cref{ssec:vlam}. With these $\CH_\alpha$-modules in hand, we
can state and prove \cref{thm:irrep}, the main result in this
section. Finally, consequences for semisimple specializations of $\CH$ are
given in \cref{cor:sp} and \cref{cor:K}.

\subsection{Pseudo-compositions and orbit
  representatives} \label{ssec:oreps} Because $D \cong \mu_b^n$, we have
$X(D)\cong X(\mu_b^n)\cong X(\mu_b)^n$ and so we may use $[b]^n$
to index the characters of $D$. We make this identification precise and
choose the usual orbit representatives for the $W$ action as follows.

For $(i_1, \dots, i_n)\in [b]^n$ define
\[ 
\chi_{(i_1, \dots, i_n)} \colon D\to A\quad\text{by}\quad
\chi_{(i_1, \dots, i_n)}\left(\begin{bmatrix} \zeta^{p_1}&& \\ &\ddots& \\
    &&\zeta^{p_n} \end{bmatrix}\right) = \zeta^{i_1p_1+\dotsm +i_np_n} .
\]
Then $X(D)= \{\, \chi_{\ibar} \mid \ibar \in [b]^n\,\}$.  For example,
$\chi_{(1, \dots, 1)}$ is the determinant character and $\chi_{(b, \dots,
  b)}$ is the trivial character.  Notice that
\begin{equation}\label{eq:te}
  T_1^{p_1} \dots T_n^{p_n} e_{\chi_{\ibar}} = \zeta^{i_1p_1+\dotsm
    +i_np_n} e_{\chi_{\ibar}}.
\end{equation}

When considered as a permutation group, $W$ acts on $[b]^n$ on the right by
place permutation: for $w\in W$ and $\ibar=(i_1, \dots, i_n)\in [b]^n$,
$\ibar\cdot w= (i_{w(1)}, \dots, i_{w(n)})$. The proof of the next lemma is
straightforward and is omitted.

\begin{lemma}\label{lem:wact}
  Suppose $w\in W$ and $\ibar\in [b]^n$. Then
  \[
  w\cdot \chi_{\ibar} = \chi_{\ibar\cdot w\inverse}.
  \]
\end{lemma}

It follows from the lemma that $\chi_{\ibar}$ and $\chi_{\jbar}$ lie in the
same $W$-orbit if and only if $\jbar$ can be obtained from $\ibar$ by
permuting the entries.

\subsection{}
A pseudo-composition of $n$ with $b$ parts, sometimes called a
$b$-composition of $n$, is a $b$-tuple $(m_1, \dots, m_b)$ of non-negative
integers such that $m_1+\dotsb +m_b=n$. Let $\compnb nb$ denote the set of
pseudo-compositions of $n$ with $b$ parts.

Define
\[
\pi\colon [b]^n\to \compnb nb \quad\text{by} \quad \pi(i_1, \dots, i_n)=
(m_1, \dots, m_b),
\]
where for $j \in[b]$, $m_j=|\{\, l\mid i_l=j\,\}|$ is the multiplicity of
$j$ in the tuple $(i_1, \dots, i_n)$. Then $\pi$ is an orbit map for the
action of $W$ on $[b]^n$ and $\chi_{\ibar}$ and $\chi_{\jbar}$ lie in the
same $W$-orbit if and only if $\pi(\ibar)= \pi(\jbar)$. Abusing notation
slightly, define
\[
\pi\colon X(D)\to \compnb nb \quad \text{by}\quad \pi(\chi_{\ibar}) =
\pi(\ibar).
\]
Then, $\chi$ and $\chi'$ lie in the same $W$-orbit if and only if
$\pi(\chi)= \pi(\chi')$. From now on we use the map $\pi$ to identify the
set $\CC$ of $W$-orbits in $X(D)$ with the set $\compnb nb$.

For $\alpha=(m_1, \dots, m_b)\in \compnb nb$ define
\[
\ibar_\alpha=( \underbrace{1, \dots, 1}_{m_1}, \dots, \underbrace{b, \dots,
  b}_{m_b} )\in [b]^n\quad \text{and}\quad \chi_\alpha=
\chi_{\ibar_\alpha}\in X(D).
\]
For example, if $\alpha = (0,2,3,2) \in \compnb 74$, then $\ibar_\alpha =
(2,2,3,3,3,4,4)\in [4]^7$, and if $\beta=(0,0,0,7) \in \compnb 74$, then
$\chi_\beta$ is the trivial character. It is easy to see that $\{\,
\ibar_\alpha\mid \alpha\in \compnb nb\,\}$ is the set of all non-decreasing
tuples in $[b]^n$ and is a complete set of orbit representatives for the
action of $W$ on $[b]^n$. Thus, $\{\, \chi_\alpha\mid \alpha \in \compnb
nb\,\}$ is a complete set of orbit representatives in $X(D)$.

\subsection{Multipartitions and tableaux}\label{ssec:mt}
A partition of $n$ is a tuple $\lambda=(n_1, \dots, n_p)$ of positive
integers such that $n_1\geq \dotsm \geq n_p$ and $n=n_1+\dotsb+n_p$. The
integers $n_i$ are the parts of $\lambda$ and $|\lambda|=n$ is the size of
$\lambda$. By definition, the empty tuple is a partition of $0$ called the
empty partition and denoted by $\emptyset$. We have $|\emptyset|=0$. A
partition is a partition of a non-negative integer.

A $b$-partition of $n$, sometimes called a $b$-multipartition of $n$, is a
$b$-tuple $\lambda= (\lambda^1, \dots, \lambda^b)$ of partitions, some of
which could be the empty partition, such that
$|\lambda^1|+ \dotsm +|\lambda^b| = n$. Let $\bpart bn$ denote the set of
$b$-partitions of $n$.

For $\lambda= (\lambda^1, \dots, \lambda^b)$, a $b$-partition of $n$, define
\[
|\lambda| =(|\lambda^1|, \dots,|\lambda^b|).
\]
Then clearly $|\lambda|\in \compnb nb$. Note that a $b$-partition of $n$ is
given by a pseudo-composition of $n$ with $b$ parts, say $\alpha=(m_1,
\dots, m_b)$, together with a partition of $m_i$ for each $i\in[b]$.

A partition of $n$ may be visualized as a Young diagram and we frequently
identify a partition with its corresponding Young diagram without comment.
Suppose $\lambda = (\lambda^1, \dots, \lambda^b)\in \bpart bn$.  The Young
diagram of $\lambda$ is the $b$-tuple of Young diagrams whose $i\th$ entry
is the Young diagram of $\lambda^i$. 

A Young $b$-tableau with shape $\lambda$ is a bijection between
$\{1, \dots, n\}$ and the boxes in the Young diagram of $\lambda$. Young
$b$-tableaux with shape $\lambda$ are usually visualized by filling in the
boxes in the Young diagram of $\lambda$ with the integers $1$, \dots,
$n$. Let $\YT^\lambda$ denote the set of Young $b$-tableaux with shape
$\lambda$. If $\tau =(\tau^1, \dots, \tau^b)$ is a Young $b$-tableau with
shape $\lambda$, then $\tau$ is standard if the entries in each row and
column of $\tau^i$ are increasing for each $i \in[b]$. Let $\SYT^\lambda$
denote the set of standard Young $b$-tableaux with shape $\lambda$.

For example, if
\[
\ytableausetup{smalltableaux, centertableaux} \lambda = ((2,1), \emptyset,
(3,2), (1,1))= \left(\, \ydiagram{2,1}, \emptyset, \ydiagram{3,2},
  \ydiagram{1,1}\,\right)\in \bpart 4{10},
\]
then
\[
|\lambda|=(3,0,5,2)\in \compnb {10}4 \quad\text{and}\quad \tau =\left(\,
  \begin{ytableau} 3&7\\9 \end{ytableau}, \emptyset,
  \begin{ytableau} 2 & 5 & 10\\ 4&8 \end{ytableau},
  \begin{ytableau} 1\\6\end{ytableau} \,\right) \in \SYT^\lambda.
\]

Suppose $\lambda=(\lambda^1, \dots, \lambda^b)\in \bpart bn$ and $\tau=
(\tau^1, \dots, \tau^b)\in \YT^\lambda$. For $i\in [n]$, define $\tau_i =j$
if $i$ appears in $\tau^j$. In other words, $\tau_i=j$ if the box containing
$i$ in the Young diagram of $\lambda$ is in $\lambda^j$. By construction,
the tuple $(\tau_1, \dots, \tau_n)$ is in $[b]^n$. Define
\[
\xi\colon \YT \to [b]^n\quad\text{by}\quad \xi(\tau)= (\tau_1, \dots,
\tau_n).
\]
Continuing the example above we have
\[
\xi \left(\left(\,
    \begin{ytableau} 3&7\\9 \end{ytableau}, \emptyset,
    \begin{ytableau} 2 & 5 & 10\\ 4&8 \end{ytableau},
    \begin{ytableau} 1\\6\end{ytableau} \,\right) \right)
=(4,3,1,3,3,4,1,3,1,3 ) \in [4]^{10} .
\]

Next, define
\[
\YT^\lambda_0=\{\, \tau\in \YT^\lambda \mid \tau_1\leq \dots \leq \tau_n
\,\} \quad \text{and} \quad \SYT^\lambda_0=\YT^\lambda_0\cap \SYT.
\]
Thus, if $|\lambda|=(m_1, \dots, m_b)$, then $\tau=(\tau^1, \dots,
\tau^b)\in \YT^\lambda_0$ if and only if
\[
\tau_1= \dots= \tau_{m_1} =1,\ \tau_{m_1+1}= \dots= \tau_{m_1+m_2} =
2,\text{ and so on}.
\]
For example,
\[
\left(\, \begin{ytableau} 1&3\\2 \end{ytableau}, \emptyset,
  \begin{ytableau} 4 & 6 & 7\\ 5&8 \end{ytableau},
  \begin{ytableau} 9\\10\end{ytableau} \,\right) \in \SYT^\lambda_0,
\]
where as above $\lambda= ((2,1), \emptyset, (3,2), (1,1))$.

As a matter of convention, we identify $1$-partitions with partitions and
$1$-tableaux with tableaux. For example, $\CP(m)=\bpart 1m$ is the set of
partitions of $m$, and if $\hat \lambda\in \CP(m)$, then
$\SYT^{\hat \lambda}$ denotes the set of standard Young tableaux with shape
$\hat \lambda$.

\subsection{}\label{ssec:wat}
Suppose $\lambda\in \bpart bn$. The group $W$ acts on $\YT^\lambda$ in the
obvious way: for $w\in W$ and $\tau\in \YT^\lambda$, $w \cdot \tau$ is the
Young $b$-tableau obtained by applying the permutation $w$ to the entries of
$\tau$. Define
\[
\chi_\tau = \chi_{\xi(\tau)}\in X(D).
\]

\begin{lemma}\label{lem:wt}
  Suppose $\tau\in \YT^\lambda$. Then $w\cdot \chi_\tau = \chi_{w\cdot
    \tau}$.
\end{lemma}

\begin{proof}
  Note that the box that contains $w\inverse(j)$ in $\tau$ is the same as
  the box that contains $j$ in $w\cdot \tau$, so
  $\tau_{w\inverse(j)} = (w\cdot \tau)_j$. Therefore
  \[
  w\cdot \chi_\tau= w\cdot \chi_{(\tau_1, \dots, \tau_n)} =
  \chi_{({\tau_{w\inverse(1)}, \dots, \tau_{w\inverse(n)}})} = \chi_{w\cdot
    \tau}
  \]
  by \cref{lem:wact}.
\end{proof}

\subsection{}\label{ssec:sum}
To summarize, given $\lambda\in \bpart bn$, $\tau\in \YT^\lambda$, and
$\alpha\in \compnb nb$, we have characters
\[
\chi_\tau= \chi_{\xi(\tau)}, \quad \chi_\alpha= \chi_{\ibar_\alpha},
\quad\text{and}\quad \chi_{|\lambda|} \quad \text{in}\quad X(D).
\]
It is easy to see that $\pi(\xi(\tau))= |\lambda|$ and so $\chi_\tau$ is in
the $W$-orbit of $\chi_{|\lambda|}$. In addition, if $\tau_0\in
\YT^\lambda_0$ and $|\lambda|= \alpha$, then $\xi(\tau_0)= \ibar_\alpha$ and
\begin{equation}
  \label{eq:sum}
  \chi_{\tau_0}= \chi_{|\lambda|}= \chi_\alpha.
\end{equation}

\subsection{A factorization of \texorpdfstring{$\mathbf
    b$}{b}-tableaux}\label{ssec:wa} 
Now suppose $\alpha=(m_1, \dots, m_b)\in \compnb nb$. Define $\mbar_0=0$,
and for $j\in [b]$ define $\mbar_j= m_1+\dots +m_j$ and $M_j= [\mbar_j]
\setminus [\mbar_{j-1}]$.  Define the Young subgroup $W_{\alpha}$ of $W$ by
\[
W_{\alpha} =\{\, w\in W \mid \forall\, j\in [b],\ w(M_j)= M_j\,\}.
\]
It is easy to see that $W_\alpha$ is the stabilizer of $\chi_\alpha$ in $W$
and that $W_{\alpha} \cong \Sigma_{m_1}\times \dotsm \times \Sigma_{m_b}$,
where $\Sigma_0$ is understood to be the trivial group. Let $W^{\alpha}$
denote the set of minimal length left coset representatives of $W_{\alpha}$
in $W$. Considering $w\in W$ as a permutation, that is, as a bijection
$w\colon [n] \to [n]$, one may consider the restriction of $w$ to each $M_j$
as a function $w|_{M_j} \colon M_j\to [n]$, and then
\[
W^\alpha=\{\, w\in W\mid \text{$\forall\, j\in[b]$, $w|_{M_j}$ is
  increasing}\,\}.
\]
It is well-known that the multiplication map $W^{\alpha} \times
W_{\alpha}\to W$ is a bijection.

Suppose $\lambda=(\lambda^1, \dots, \lambda^b)\in \bpart bn$. Then
$|\lambda|\in \compnb nb$ and so $W_{|\lambda|}$ and $W^{|\lambda|}$ are
defined.  Let $\varphi\colon W^{|\lambda|} \times \SYT^\lambda_0\to
\YT^\lambda$ be the action map defined by $\varphi(w, \tau_0)= w\cdot
\tau_0$.

\begin{proposition}\label{pro:factab}
  The mapping $\varphi\colon W^{|\lambda|} \times \SYT^\lambda_0\to
  \YT^\lambda$ is injective with image equal to $\SYT^\lambda$ and so the
  $W$-action on $\YT^\lambda$ induces a bijection between $W^{|\lambda|}
  \times \SYT^\lambda_0$ and $\SYT^\lambda$. Therefore, given $\tau\in
  \SYT^\lambda$, there is a unique $w\in W^{|\lambda|}$ and a unique
  $\tau_0\in \SYT^\lambda_0$ such that $\tau =w \cdot \tau_0$.
\end{proposition}

\begin{proof}
  Suppose $|\lambda|= (m_1, \dots, m_b)$ and that for $j\in [b]$, $\mbar_j$
  and $M_j$ are defined as in \cref{ssec:wa}. Notice that $\tau=(\tau^1,
  \dots, \tau^b)\in \YT^\lambda_0$ if and only if for all $j\in [b]$, $M_j$
  is the set of entries of $\tau^j$.

  To show that $\varphi$ is injective, suppose $w, w'\in W^{|\lambda|}$,
  $\tau_0, \tau_0'\in \SYT^\lambda_0$, and
  $w\cdot \tau_0 = w'\cdot \tau_0'$. Then
  $w\inverse w'\cdot \tau_0'= \tau_0$, so $w\inverse w'(M_j)=M_j$ for all
  $j\in [b]$. But then $w\inverse w'\in W_{|\lambda|}$, so $w=w'$, which
  implies that $\tau_0=\tau_0'$ as well.

  Because $w\in W^{|\lambda|}$ if and only if the restriction of $w$ to
  $M_j$ is an increasing function (when $w$ is considered as a permutation),
  and a standard Young $b$-tableau $\tau_0=(\tau_0^1, \dots, \tau_0^b)$ with
  shape $\lambda$ is in $\SYT^\lambda_0$ if and only if $M_j$ is the set of
  entries of $\tau_0^j$ for $j\in [b]$, it is clear that if $w\in
  W^{|\lambda|}$ and $\tau_0\in \SYT^\lambda_0$, then $w\cdot \tau_0$ is
  standard. Thus, the image of $\varphi$ is contained in $\SYT^\lambda$.

  To show that the image of $\varphi$ is equal to $\SYT^\lambda$, suppose
  $\tau=(\tau^1, \dots, \tau^b)\in \SYT^\lambda$. For $j\in[b]$ let $N_j =
  \{i^j_1, \dots, i^j_{m_j}\}$ be the set of entries of $\tau^j$, where
  $i^j_1< \dotsm< i^j_{m_j}$. Define a permutation $w\in W$ by
  \[
  w(\mbar_{j-1}+l)= i^j_l \quad\text{for}\quad j\in[b] \quad\text{and}\quad
  l\in[m_j] .
  \]
  Then $w\in W^{|\lambda|}$. Moreover, $w\inverse(N_j)=M_j$ and the
  restriction of $w\inverse$ to $N_j$ is increasing for all $j\in [b]$. This
  implies that $w\inverse\cdot \tau \in \YT^\lambda_0$ and $w\inverse\cdot
  \tau$ is standard, so $w\inverse\cdot \tau\in \SYT^\lambda_0$. Clearly
  $\varphi(w, w\inverse \cdot \tau)=\tau$ and hence the image of $\varphi$
  is equal to $\SYT^\lambda$.
\end{proof}

\subsection{Hoefsmit's modules for Iwahori-Hecke algebras of type
  \texorpdfstring{$\mathbf A$}{A}}\label{ssec:hm}
Suppose $m>1$ is an integer.  The Iwahori-Hecke algebra of type $A_{m-1}$
over $A$ with parameters $\bq$ and $b\ba$ is the unital $A$-algebra
$\CI^m$ with generators
\[
\text{$S_1$, \dots, $S_{m-1}$}
\]
and relations 
\begin{alignat*}{2}
  &S_{i}S_{j}=S_{j}S_{i} \qquad&& \text{for $i,j\in[ m-1]$ with
    $|i-j|>1$,} \\
  &S_{i}S_{i+1}S_{i}= S_{i+1}S_{i}S_{i+1} \qquad&& \text{for $i\in[m-2]$,
    and} \\ 
  &S_{i}^2=\bq +b\ba S_i \qquad&& \text{for $i\in[m-1]$.}
\end{alignat*} 
It is shown in \cite[\S IV.2 Exercise 24]{bourbaki:groupes} that $\dim \CI^m
=m!$ and that $\CI^m$ has an $A$-basis $\{\, S_\sigma\mid \sigma\in
\Sigma_m\,\}$ such that if $s_j$ is the transposition $( j\ j+1)$ in
$\Sigma_{m}$, then $S_{s_j}=S_j$ and
\[
S_{s_j} S_\sigma= \begin{cases} S_{s_j\sigma}
  &\text{if $\ell(s_j\sigma) =  \ell(\sigma) +1$} \\
  \bq S_{s_j\sigma} + b\ba S_{\sigma} &\text{if $\ell(s_j\sigma) =
    \ell(\sigma) -1$.}\end{cases}
\]
Define $\CI^0=\CI^1 =A$.

Recall the element $x\in A$ defined in \cref{ssec:AH} and that
$\sqrt{b \ba^2+4\bq}$ is a fixed square root of $b^2\ba^2+4\bq$. Define
\[
y= \frac {b\ba+\sqrt{b^2\ba^2+4\bq}} {2\bq} .
\]
Then $x=\bq y^2= b\ba y+1$ and $y$ is a unit in $A$ with
$y\inverse= \big( b\ba-\sqrt{b^2\ba^2+4\bq} \big)/2$. Notice that if
$\Shat_l\in \CI^m$ is defined by $\Shat_l=yS_l$, then
$\Shat_{l}^2=x 1+(x-1) \Shat_l$.

If $\hat \tau$ is a Young $1$-tableau and $i$ and $j$ are entries in
$\hat \tau$, then the axial distance from $i$ to $j$ in $\hat \tau$ is
\[
\delta_{\tau}(i,j)= (c_j- r_j) -(c_{i}-r_{i}),
\]
when the box in $\hat \tau$ that contains $i$ is in row $r_i$ and column
$c_i$, and similarly for $j$. For example, if $\tau= \begin{ytableau} 2 & 5
  & 10\\ 4&8 \end{ytableau}$, then $\delta_{\tau}(4,10)= (3- 1) -(1-2)=3=
-\delta_{\tau}(10,4)$.

For a partition $\hat \lambda$ of $m$, let $P^{\hat \lambda}$ be a free
$A$-module with basis indexed by $\SYT^{\hat \lambda}$, the set of standard
Young tableaux with shape $\hat \lambda$. Say
$\{\, p_{\hat \tau}\mid \hat \tau\in \SYT^{\hat \lambda}\,\}$ is an
$A$-basis of $P^{\hat \lambda}$. Hoefsmit \cite{hoefsmit:representations}
has shown that there is an $\CI^{m}$-module structure on $P^{\hat \lambda}$
such that for $l\in[m-1]$ and $\hat \tau\in \SYT^{\hat \lambda}$,
\begin{equation}
  \label{eq:iab}
  S_l \cdot p_{\hat \tau}= \frac {x-1}{y(1-x^k)}\, p_{\hat \tau} +  
  \frac {x-x^k}{y(1-x^k)} \, p_{s_l\cdot \hat \tau},  
\end{equation}
where $k= \delta_{\hat \tau}(l+1,l)$ and $p_{s_l\cdot \hat \tau}$ is taken
to be zero if $s_l\cdot \hat \tau$ is not standard.

\subsection{The algebra \texorpdfstring{$\mathbf{\CH_{\alpha}}$}{Ha} and
  \texorpdfstring{$\mathbf{\CH_{\alpha}}$} {Ha}-modules}\label{ssec:Ha}
We now take up the structure of the algebra $\CH_\alpha$ and the
construction of $\CH_\alpha$-modules. From this subsection until the
statement of \cref{thm:irrep}, $\alpha=(m_1, \dots, m_b)\in \compnb nb$
denotes a fixed pseudo-composition of $n$ with $b$ parts.

\begin{lemma}\label{lem:alba} 
  Consider the subalgebra $\CH_\alpha$ of $\CH$ and the right
  $\CH_\alpha$-module $\CH e_{\chi_\alpha}$.
  \begin{enumerate}
  \item $\{\, t_w e_{\chi_\alpha} \mid w\in W_{\alpha}\,\}$ is an $A$-basis
    of $\CH_{\alpha}$ and so $\CH_\alpha$ is a free $A$-module with rank
    $|W_\alpha|$. \label[equation]{i:1}
  \item $\{\, t_w e_{\chi_\alpha} \mid w\in W^\alpha\,\}$ is an
    $\CH_\alpha$-basis of $\CH e_{\chi_\alpha}$ and so $\CH e_{\chi_\alpha}$
    is a free $\CH_\alpha$-module with rank
    $|W^\alpha|$. \label[equation]{i:2}
  \end{enumerate}
\end{lemma}

\begin{proof}
  Because $\{\, t_we_{\chi} \mid \ w\in W,\ \chi\in X(D)\,\}$ is an
  $A$-basis of $\CH$, the set $\{\, e_{\chi_\alpha} t_w e_{\chi}
  e_{\chi_\alpha} \mid \ w\in W, \chi \in X(D) \,\}$ spans
  $\CH_\alpha$. If $\chi\ne \chi_\alpha$, then $e_{\chi_\alpha} e_{\chi}=0$
  and so the set $\{\, e_{\chi_\alpha} t_w e_{\chi_\alpha} \mid w\in W\,\}$
  spans $\CH_\alpha$. Since $e_{\chi_\alpha} t_w e_{\chi_\alpha} =
  e_{\chi_\alpha} e_{w\cdot \chi_\alpha} t_w$, it follows that
  $e_{\chi_\alpha} t_w e_{\chi_\alpha} =0$ if $w\cdot\chi_\alpha \ne
  \chi_\alpha$. If $w\cdot \chi_\alpha= \chi_\alpha$, then $e_{\chi_\alpha}
  t_w e_{\chi_\alpha} = t_w e_{\chi_\alpha}$.  Thus, the set $\{\, t_w
  e_{\chi_\alpha} \mid w\in W_{\alpha}\,\}$ spans $\CH_\alpha$. Since this
  set is a subset of a basis of $\CH$, it is linearly independent and so it
  is a basis of $\CH_\alpha$.

  Clearly $\CH e_{\chi_\alpha}$ is a free $A$-module with basis
  \[
  \{\, t_we_{\chi_\alpha} \mid w\in W\,\}= \{\, t_{w_1} e_{\chi_\alpha} (
  t_{w_2} e_{\chi_\alpha}) \mid w_1\in W^\alpha,\ w_2\in W_\alpha\,\}.
  \]
  It follows from this observation and what has already been proved that the
  set $\{\, t_we_{\chi_\alpha} \mid w\in W^\alpha\,\}$ is an
  $\CH_\alpha$-basis of $\CH e_{\chi_\alpha}$.
\end{proof}

\subsection{}
To continue, we need to make the isomorphism $W_\alpha \cong \Sigma_{m_1}
\times \dotsm \times \Sigma_{m_b}$ precise. Recall the integers $\mbar_j$
and the sets $M_j$ defined in terms of $\alpha$ in \cref{ssec:wa}. Suppose
$j\in[b]$ and $\sigma\in \Sigma_{m_j}$. Define a permutation $w_\sigma^j\in
W$ by
\[
w_\sigma^j (k) =
\begin{cases}
  k&\text{if $k\notin M_j$ and} \\
  \mbar_{j-1}+ \sigma(l) & \text{if $k=\mbar_{j-1}+l\in M_j$.}
  \end{cases}
\]
Then the rule $\sigma\mapsto w_\sigma^j$ defines an injective group
homomorphism from $\Sigma_{m_j}$ to $W_{\alpha}$. Let $W_{\alpha,j}$ denote
the image of this group homomorphism, so $W_{\alpha,j} \cong
\Sigma_{m_j}$. Then $W_{\alpha} = W_{\alpha,1} \dotsm W_{\alpha,b}$ is the
internal direct product of the subgroups $W_{\alpha,j}$.

In the following, we use two more notational conventions. First, because of
the factor $T_i^{(b^2-b)/2}$ in relation \cref{rel:6}, the parity of $b$
will play a role in the rest of the proof of the paper. To simplify the
notation, define $\epsilon= \zeta^{(b^2-b)/2}$. Then
\[
\text{$\epsilon=1$ if $b$ is odd} \qquad\text{and}\qquad \text{$\epsilon=-1$
  if $b$ is even}.
\]

Second, because $\{M_1, \dots, M_b\}$ is a partition of the set $[n]$, for
$i\in [n]$, there is a unique $j\in [b]$ such that $i\in M_j$. Define
$i'=j$. In other words, $(\mbar_{j-1}+l)'=j$ for $j\in [b]$ and
$l\in [m_j]$.

With these conventions, for $s_i\in W_\alpha$ define
\[
\Stilde^\alpha_i= \epsilon^{i'} t_{s_i}e_{\chi_\alpha} \in \CH_\alpha.
\]

\begin{lemma}\label{lem:alpr}
  The elements $\Stilde^\alpha_i$ for $s_i\in W_\alpha$ generate the
  $A$-algebra $\CH_\alpha$ and satisfy the relations
  \begin{align}
    &\Stilde^\alpha_{i_1} \Stilde^\alpha_{i_2} = \Stilde^\alpha_{i_2}
      \Stilde^\alpha_{i_1} && s_{i_1},s_{i_2}\in W_{\alpha}, \
                              |i_1-i_2|>1, \tag{$\textrm b_1$} \label{rel:b1}\\
    &\Stilde^\alpha_{i} \Stilde^\alpha_{i+1} \Stilde^\alpha_{i} =
      \Stilde^\alpha_{i+1} \Stilde^\alpha_{i} \Stilde^\alpha_{i+1}&& 
    s_i, s_{i+1}\in W_{\alpha},\quad \text{and} \tag{$\textrm 
    b_2$} \label{rel:b2}\\
    &(\Stilde^\alpha_{i})^2 =\bq e_{\chi_\alpha} + b\ba\,
      \Stilde^\alpha_{i}&& s_{i}\in W_{\alpha}. \tag{$\textrm
                           q$} \label{rel:q}
  \end{align}
\end{lemma}

\begin{proof}
  It follows from \cref{thm:rel} and equation \cref{ssec:ciw}\cref{eq:we1}
  that the subalgebra generated by
  $\{\, \Stilde^\alpha_i \mid s_i\in W_\alpha \,\}$ contains the basis
  $\{\, t_w e_{\chi_\alpha} \mid w\in W_\alpha\,\}$ of $\CH_\alpha$ and
  hence is equal $\CH_\alpha$.

  Relations \cref{rel:b1} and \cref{rel:b2} follow immediately from
  relations \cref{rel:4} and \cref{rel:5} because $t_{s_l}$ and
  $e_{\chi_\alpha}$ commute when $s_l\in W_\alpha$.

  To prove relation \cref{rel:q}, using relation \cref{rel:6} and that
  $s_i\in W_{\chi_\alpha}$ we have
  \begin{align*}
    (\Stilde^\alpha_{i})^2&= (\epsilon^{i'} t_{s_i}e_{\chi_\alpha})^2\\
    &= t_{s_i}^2 e_{\chi_\alpha} \\
    &=\bigg(\bq +\ba T_{i}^{(b^2-b)/2} \bigg(\sum_{k=0}^{b-1}
    T_i^k T_{i+1}^{-k}\bigg) R_i\bigg) e_{\chi_\alpha} \\
    &=\bq e_{\chi_\alpha} +\ba T_{i}^{(b^2-b)/2} \bigg( \sum_{k=0}^{b-1}
    T_i^k T_{i+1}^{-k}\bigg) e_{\chi_\alpha} t_{s_i}.
  \end{align*}
  Notice that $s_i\in W_\alpha$ implies that $(i+1)'=i'$ and so by
  \cref{ssec:oreps}\cref{eq:te} we have
  \[
  T_i^k T_{i+1}^{-k} e_{\chi_\alpha}= \zeta ^{i'k+i'(-k)} e_{\chi_\alpha} =
  e_{\chi_\alpha} \quad\text{for}\quad 0\leq k\leq b-1
  \]
  and
  \[
  T_{i}^{(b^2-b)/2} e_{\chi_\alpha} = \zeta^{i'(b^2-b)/2} e_{\chi_\alpha} =
  \epsilon^{i'}e_{\chi_\alpha} .
  \]
  Therefore,
  \[
  (\Stilde^\alpha_{i})^2= \bq e_{\chi_\alpha} + \epsilon^{i'} b \ba\,
  e_{\chi_\alpha} t_{s_i} = \bq e_{\chi_\alpha} + b\ba \Stilde^\alpha_{i} .
  \]
\end{proof}

The generators and relations in the lemma are in fact a presentation of
$\CH_\alpha$.

\begin{lemma}\label{lem:iiso}
  There is an $A$-algebra isomorphism 
  \[
  \eta\colon \CI^{m_1}\otimes_{A} \dotsm \otimes_{A} \CI ^{m_b}
  \xrightarrow{\ \cong\ } \CH_\alpha
  \]
  with the property that if $\rho_j\in \Sigma_{m_j}$ for $j\in[b]$, then
  $\eta(S_{\rho_1} \otimes \dotsm \otimes S_{\rho_b})= \pm
  t_we_{\chi_\alpha}$, where $w=w^1_{\rho_1} \dotsm w^b_{\rho_b}$.
\end{lemma}

\begin{proof}
  For $j\in[b]$ let $\CH_{\alpha,j}$ be the subalgebra of $\CH_\alpha$
  generated by the set $\{\, \Stilde^\alpha_i \mid s_i\in W_{\alpha,j}\,\}$
  and define
  \[
  \eta_j \colon \CI^{m_j} \to \CH_{\alpha} \quad\text{by}\quad \eta_j(S_l) =
  \Stilde^\alpha_{\mbar_{j-1}+l} = \epsilon^j t_{w^j_{\sigma_l}} e_{\chi_\alpha}
  \quad \text{for $l\in[m_{j}]$.}
  \]
  Then it follows from \cref{lem:alba}\cref{i:1} and \cref{lem:alpr} that
  $\eta_j$ induces an $A$-algebra isomorphism
  $\CI^{m_j} \cong \CH_{\alpha,j}$. If $\rho\in \Sigma_{m_j}$ and
  $\tau_{l_1} \dotsm \tau_{l_r}$ is a reduced expression for $\rho$, then a
  straightforward computation shows that
  \begin{equation}\label{eq:eta}
  \eta_j(S_\rho) = (\epsilon^j)^r t_{w^j_\rho} e_{\chi_\alpha}.  
  \end{equation}

  The homomorphisms $\eta_j$ for $j\in [b]$ determine an $A$-linear map
  \[
  \eta\colon \CI^{m_1} \otimes_{A} \dotsm \otimes_a \CI ^{m_b} \to
  \CH_\alpha \quad\text{with}\quad \eta(h_1\otimes \dotsm \otimes h_b)=
  \eta_1(h_1) \dotsm \eta_b(h_b).
  \]
  The map $\eta$ is an $A$-algebra homomorphism because each $\eta_j$ is an
  $A$-algebra homomorphism and because $\CH_{\alpha,j}$ and
  $\CH_{\alpha,j'}$ commute elementwise for $j,j'\in[b]$ with $j\ne j'$.
  Because $\eta$ is an $A$-algebra homomorphism, if $\rho_j\in \Sigma_{m_j}$
  for $j\in[b]$ and $w=w^1_{\rho_1} \dotsm w^b_{\rho_b}$, then by
  \cref{eq:eta}
  $\eta(S_{\rho_1} \otimes \dotsm \otimes S_{\rho_b}) = \pm
  t_we_{\chi_\alpha}$.
  Therefore, $\eta$ maps the basis
  $\{\, S_{\rho_1} \otimes \dotsm \otimes S_{\rho_b}\mid \forall j\in[b],\
  \rho_j\in \Sigma_{m_j}\,\}$
  of $\CI^{m_1} \otimes_{A} \dotsm \otimes_{A} \CI^{m_b}$ bijectively onto a
  basis of $\CH_\alpha$, and hence is an $A$-algebra isomorphism.
\end{proof}

\subsection{The \texorpdfstring{$\mathbf{\CH}$}{H}-module
  \texorpdfstring{$\mathbf{V^\lambda}$}{V}}\label{ssec:vlam}
Suppose $\lambda=(\lambda^1, \dots, \lambda^b)\in \bpart bn$ and $|\lambda|=
\alpha$. Recall from \cref{ssec:hm} that for $j\in [b]$, $P^{\lambda^j}$ is
an $\CI^{m_j}$-module with $A$-basis $\{\, p_{\hat \tau}\mid \hat \tau\in
\SYT^{\lambda^j}\,\}$ and the action of the generator $S_l\in \CI^{m_j}$ on
the basis element $p_{\hat \tau}$ given by \cref{ssec:hm}\cref{eq:iab}.

Define
\[
P^\lambda_0= P^{\lambda^1} \otimes_{A} \dotsm \otimes_{A} P^{\lambda^b}.
\]
Then $P^\lambda_0$ is an
$\CI^{m_1} \otimes_{A} \dotsm \otimes_{A} \CI^{m_b}$-module.  We consider
$P^\lambda_0$ as an $\CH_\alpha$-module via transport of structure across
the isomorphism
$\eta\colon \CI^{m_1} \otimes_{A} \dotsm \otimes_{A} \CI^{m_b}
\xrightarrow{\ \cong\ } \CH_\alpha$ in \cref{lem:iiso}.

Finally, define
\[
V^\lambda= F_\alpha (P^\lambda_0)= \CH\otimes_{\CH_\alpha} P^\lambda_0 = \CH
e_{\chi_\alpha} \otimes_{\CH_\alpha} P^\lambda_0.
\]

We can now state the main theorem in this section.

\begin{theorem}\label{thm:irrep}
  Suppose $\lambda\in \bpart bn$ and set $\alpha=|\lambda|\in \compnb
  nb$. The $\CH$-module $V^\lambda$ has the following properties.
  \begin{enumerate}
  \item $V^\lambda$ is a free $A$-module with a basis $\{\, v_\tau\mid \tau
    \in \SYT^\lambda\,\}$ indexed by $\SYT^\lambda$. \label[equation]{i:t1}
  \item Suppose $\tau\in \SYT^\lambda$ and that $\tau = w\cdot \tau_0$ with
    $w\in W^{\alpha}$ and $\tau_0\in \SYT^\lambda_0$. Then $v_\tau = t_w
    v_{\tau_0}$. \label[equation]{i:t2}
  \item For $w\in W^\alpha$, the weight space
    $V^\lambda_{w\cdot \chi_{\alpha}}$ is the $A$-span of
    $\{\, t_w v_{\tau_0} \mid \tau_0\in \SYT^\lambda_0\,\}$. In particular,
    $V^\lambda_{\chi_{\alpha}}$ is the $A$-span of
    $\{\, v_{\tau_0} \mid \tau_0\in \SYT^\lambda_0\,\}$ and
    $V^\lambda_{\chi_{\alpha}}\cong P^\lambda_0$. \label[equation]{i:t3}
  \item $V^\lambda$ is in $\CO_{\alpha}$. \label[equation]{i:t4}
  \item Suppose $\tau\in \SYT^\lambda$. Then for $d\in D$
    \[
    t_d\cdot v_\tau = \chi_\tau(d) v_\tau,
    \]
    and for $i\in [n-1]$
    \begin{equation*}
      \label{eq:siact}
      t_{s_i} v_\tau =
      \begin{cases}
        v_{s_i\cdot \tau}&\text{if $j<j'$}\\
        \epsilon^j \big(\frac{1-x} {y(1-x^k)}\big) v_{\tau} + \epsilon^j
        \big(\frac {x-x^k} {y(1-x^k)}\big) v_{s_i \cdot\tau}
        &\text{if $j=j'$}\\
        \bq\, v_{s_i\cdot \tau}&\text{if $j>j'$,}
      \end{cases}
    \end{equation*}
    where $j=\tau_i$, $j'=\tau_{i+1}$, $k=\rho_{\tau^j}(i+1,i)$, and $v_{s_i
      \cdot\tau}=0$ if $s_i \cdot\tau$ is not
    standard. \label[equation]{i:t5}
  \item Let $f\colon A\to k$ be a specialization. Then ${}_fV^\lambda$ is an
    irreducible ${}_f\CH$-module if and only if ${}_fP^\lambda_0$ is an
    irreducible ${}_f\CH_{\alpha}$-module. \label[equation]{i:t6}
  \end{enumerate}
\end{theorem}

\begin{proof}
  By construction $\{\, p_{\hat \tau^1} \otimes \dotsm \otimes p_{\hat
    \tau^b}\mid \forall \ j\in [b], \hat \tau^j\in \SYT^{\lambda^j} \,\}$ is
  an $A$-basis of $P^\lambda_0$. Given a tuple $(\hat \tau^1,\dots, \hat
  \tau^b)$ with $\hat \tau^j\in \SYT^{\lambda^j}$ for $j\in [b]$, define
  $\tau^j_0$ to be the tableau obtained from $\hat \tau^j$ by adding
  $\mbar_{j-1}$ to each entry and define $\tau_0= (\tau^1_0,\dots,
  \tau^b_0)$. Then $\tau_0\in \SYT^\lambda_0$. Define
  \[
  \vtilde_{\tau_0}= p_{\hat \tau^1} \otimes \dotsm \otimes p_{\hat \tau^b}.
  \]
  Then $\{\, \vtilde_{\tau_0}\mid \tau_0\in \SYT^\lambda_0\,\}$ is an
  $A$-basis of $P^\lambda_0$ and it follows from \cref{lem:alba}\,\cref{i:2}
  that $\{\, t_w\otimes \vtilde_{\tau_0} \mid w\in W^\alpha,\
  \tau_0\in\SYT^\lambda_0\,\}$ is an $A$-basis of $V^\lambda$. For $w\in
  W^\alpha$ and $\tau_0\in \SYT^\lambda_0$ define
  \[
  v_\tau= t_w\otimes \vtilde_{\tau_0}.
  \]
  Then $\{\, v_\tau \mid \tau\in\SYT^\lambda\,\}$ is an $A$-basis of
  $V^\lambda$. This proves \cref{i:t1}. By definition, $v_\tau = t_w\otimes
  \vtilde_{\tau_0}= t_w(1\otimes \vtilde_{\tau_0}) = t_wv_{\tau_0}$ and so
  \cref{i:t2} holds as well.

  Next, for $\tau_0\in \SYT^\lambda_0$,
  $e_{\chi_\alpha} v_{\tau_0}= e_{\chi_\alpha} (1\otimes \vtilde_{\tau_0}) =
  1\otimes e_{\chi_\alpha} \vtilde_{\tau_0} = 1\otimes \vtilde_{\tau_0}
  =v_{\tau_0}$.
  This implies that $v_{\tau_0}\in V^\lambda_{\chi_\alpha}$. Therefore, it
  follows from \cref{ssec:ciw}\cref{eq:we2} that
  $t_wv_{\tau_0}\subseteq V^\lambda_{w\cdot \chi_\alpha}$ for every
  $w\in W$. In particular, if $\tau_0\in \SYT^\lambda_0$ and
  $w\in W^\alpha$, then
  $v_\tau= t_wv_{\tau_0} \in V^\lambda_{w\cdot \chi_\alpha}$. Because
  $\{\, v_\tau\mid \tau\in \SYT^\lambda\,\}$ is a basis of $V^\lambda$ this
  shows that for $w\in W^\alpha$, $V^\lambda_{w\cdot \chi_\alpha}$ is the
  $A$-span of $\{\, t_wv_{\tau_0}\mid \tau_0\in \SYT^\lambda_0\,\}$. In
  particular, $V^\lambda_{\chi_\alpha}$ is the $A$-span of
  $\{\, v_{\tau_0}\mid \tau_0\in \SYT^\lambda_0\,\}$ and the rule
  $\vtilde_{\tau_0}\mapsto v_{\tau_0}$ defines an $A$-linear isomorphism
  between $P^\lambda_0$ and $V^\lambda_{\chi_\alpha}$. This proves
  \cref{i:t3}. In addition, $V^\lambda_\chi=0$ if $\chi$ is not in the
  $W$-orbit of $\chi_\alpha$, and so $V^\lambda$ is in $\CO_\alpha$ as
  asserted in \cref{i:t4}.

  To prove \cref{i:t5}, fix $\tau=w\cdot \tau_0$ in $\SYT^\lambda$ with
  $w\in W^\alpha$ and $\tau_0\in \SYT^\lambda_0$. Then $v_\tau\in
  V^\lambda_{w\cdot \chi_\alpha}$, and so by \cref{ssec:ciw}\cref{eq:we2},
  \cref{ssec:sum}\cref{eq:sum}, and \cref{lem:wt} we have
  \[
  t_d\cdot v_\tau = (w\cdot \chi_{\alpha})(d)\cdot v_\tau = (w\cdot
  \chi_{\tau_0})(d)\cdot v_\tau = \chi_{w\cdot \tau_0}(d)\, v_{\tau}=
  \chi_{\tau}(d)\, v_{\tau} .
  \]

  Using the definitions and \cref{ssec:hm}\cref{eq:iab}, a straightforward
  computation shows that for $j\in [b]$ and $s_l\in W_{\alpha,j}$ we have
  \begin{equation}\label{eq:tvt}
    t_{s_l} e_{\chi_\alpha}\cdot \vtilde_{\tau_0}= \eta\inverse(t_{s_l}
    e_{\chi_\alpha})\cdot \vtilde_{\tau_0}=  \epsilon^j \left(\frac
      {x-1}{y(1-x^k)}\right)\, \vtilde_{\tau_0}+ \epsilon^j \left(\frac
      {x-x^k}{y(1-x^k)} \right) \, \vtilde_{s_l\cdot \tau_0},
  \end{equation}
  where $k=\delta_{\tau^j_0}(l+1,)$ and $\vtilde_{s_l\cdot\tau_0}=0$ if
  $s_l\cdot \tau_0$ is not standard.

  Now suppose $i\in [n-1]$ and consider $t_{s_i}\cdot v_\tau = t_{s_i} t_w
  \cdot v_{\tau_0}$. Suppose $\tau_i=j$ and $\tau_{i+1}=j'$. Then $i\in w(
  M_{j})$ and $i+1\in w(M_{j'})$.  There are three cases. First, if $j<j'$,
  then $s_iw>w$ and $s_iw\in W^\alpha$, so
  \[
  t_{s_i}\cdot v_\tau = t_{s_iw}\cdot v_{\tau_0} = v_{s_iw\cdot\tau_0}=
  v_{s_i\cdot \tau} .
  \]
  Second, if $j=j'$, then because $w\in W^\alpha$ and $i, i+1\in w(M_j)$,
  there is an $l$ in $M_j$ such that $l<m_j$, $w(l)=i$, and
  $w(l+1)=i+1$. Then $s_l\in W_{\alpha, j}$ and $s_iw=ws_l$. Therefore,
  using the definitions and \cref{eq:tvt} we have
  \begin{align*}
    t_{s_i}\cdot v_\tau &= t_wt_{s_l}\otimes \vtilde_{\tau_0}\\
    & = t_w\otimes t_{s_l}e_{\chi_\alpha}\cdot \vtilde_{\tau_0} \\
    & = t_w\otimes \epsilon^j \bigg( \frac {x-1}{y(1-x^k)}\,
    \vtilde_{\tau_0} + \frac {x-x^k}{y(1-x^k)}\, \vtilde_{s_l\cdot\tau_0}
    \bigg)\\
    &= \epsilon^j \bigg( \frac {x-1}{y(1-x^k)} \bigg)\,t_w\otimes
    \vtilde_{\tau_0} + \epsilon^j \bigg( \frac {x-x^k} {y(1-x^k)}\bigg) \,
    t_w \otimes \vtilde_{s_l\cdot \tau_0}\\
    &= \epsilon^j \bigg( \frac {x-1}{y(1-x^k)} \bigg)t_w\cdot v_{\tau_0} +
    \epsilon^j \bigg( \frac {x-x^k} {y(1-x^k)}\bigg) t_w v_{s_l\cdot
      \tau_0}\\
    & = \epsilon^j \bigg( \frac {x-1}{y(1-x^k)}\bigg) v_{\tau} + \epsilon^j
    \bigg( \frac {x-x^k}{y(1-x^k)}\bigg) v_{s_i\cdot \tau} ,
  \end{align*}
  where $k=\delta_{\tau^j}(i+1,i)$ and $\vtilde_{s_l\cdot \tau_0}=0= v_{s_i
    \cdot\tau}$ if $s_l\cdot \tau_0$, or equivalently if $s_i \cdot\tau$, is
  not standard. Finally, if $j>j'$, then $s_iw<w$ and $s_iw\in
  W^\alpha$. Therefore, by \cref{thm:rel}
  \begin{align*}
    t_{s_i}\cdot v_\tau &= t_{s_i}t_{w}\cdot v_{\tau_0}\\
    & = (\bq t_{s_iw} +\ba T_{i}^{(b^2-b)/2} E_i t_w )\cdot v_{\tau_0}\\
    &= \bq t_{s_iw} \cdot v_{\tau_0} +\ba T_{i}^{(b^2-b)/2} E_i t_w \cdot
    v_{\tau_0}\\
    & = \bq v_{s_i\cdot \tau} +\ba T_{i}^{(b^2-b)/2} E_i t_w e_{\chi_\alpha}
    \cdot v_{\tau_0}.
  \end{align*}
  By \cref{thm:rel} and \cref{ssec:oreps}\cref{eq:te} we have
  \begin{align*}
    E_i t_w e_{\chi_\alpha}&= \bigg(\sum_{r=0}^{b-1} T_i^r T_{i+1}^{-r}
    \bigg) t_w e_{\chi_\alpha} \\
    &= t_w \bigg( \sum_{r=0}^{b-1} T_{w\inverse(i)}^r
    T_{{w\inverse(i+1)}}^{-r} \bigg) e_{\chi_\alpha}\\
    &= \bigg( \sum_{r=0}^{b-1} (\zeta^{j-j'})^r \bigg) t_w e_{\chi_\alpha}=
    0
  \end{align*}
  because $j\ne j'$ and so $\zeta^{j-j'}$ is a $b\th$ root of unity not
  equal to $1$. Therefore,
  \[
  t_{s_i}\cdot v_\tau = \bq v_{s_i\cdot \tau}.
  \]

  The last statement in the theorem follows from \cref{thm:eqv} and the
  isomorphism ${}_fP^\lambda_0\cong {}_fV^\lambda_{\chi_\alpha}$ that
  follows from the third statement of the theorem.
\end{proof}

\subsection{Irreducible modules and semisimple
  specializations}\label{ssec:spec}

\begin{corollary}\label{cor:sp}
  Let $f\colon A\to k$ be a specialization such that $k$ is a field and
  consider the specialized algebra ${}_f\CH$.
  \begin{enumerate}
  \item Suppose that $\lambda=(\lambda^1, \dots, \lambda^b)\in \bpart bn$
    and that for all $j\in [b]$ the ${}_f\CI^{|\lambda^j|}$-module
    ${}_fP^{\lambda^j}$ is irreducible. Then the ${}_f\CH$-module ${}_f
    V^\lambda$ is absolutely irreducible.
  \item Suppose that $\alpha=(m_1, \dots, m_b)\in \compnb nb$ and that for
    all $\lambda=(\lambda^1, \dots, \lambda^b)\in \bpart bn$ with
    $|\lambda|=\alpha$ and all $j\in [b]$, the ${}_f\CI^{m_j}$-module
    ${}_fP^{\lambda^j}$ is irreducible. Then ${}_f\CH_\alpha$ is a split
    semisimple $k$-algebra and $\{\, {}_fV^\lambda\mid \lambda\in \bpart
    bn,\ |\lambda|=\alpha \,\}$ is a complete set of inequivalent,
    irreducible ${}_f\CH$-modules in ${}_f\CO_\alpha$.
  \item Suppose that for all $\lambda=(\lambda^1, \dots, \lambda^b)\in
    \bpart bn$ and all $j\in [b]$, the ${}_f\CI^{|\lambda^j|}$-module
    ${}_fP^{\lambda^j}$ is irreducible. Then ${}_f\CH$ is a split semisimple
    $k$-algebra and $\{\, {}_fV^\lambda\mid \lambda\in \bpart bn\,\}$ is a
    complete set of inequivalent, irreducible
    ${}_f\CH$-modules. \label[equation]{it:3}
  \end{enumerate}
\end{corollary}

\begin{proof}
  The algebra $\CI^m$ is a cellular algebra in the sense of Graham and
  Lehrer and $P^{\hat \lambda}$ is a cell module for every partition $\hat
  \lambda$ of $m$. Thus, it follows from \cite{grahamlehrer:cellular} that
  if ${}_fP^{\hat \lambda}$ is irreducible, then it is absolutely
  irreducible. It then follows that if ${}_fP^{\hat \lambda}$ is irreducible
  for all $\hat \lambda\in \CP(m)$, then ${}_f\CI^m$ is a split semisimple
  $k$-algebra.

  To prove the first statement, set $\alpha=|\lambda|$. By assumption,
  ${}_fP^{\lambda^j}$ is irreducible for $j\in [b]$, so ${}_fP^\lambda_0
  \cong {}_fP^{\lambda^1} \otimes_{k} \dotsm \otimes_{k} {}_fP^{\lambda^b}$
  is a product of absolutely irreducible modules and hence is absolutely
  irreducible. It follows that ${}_fV^\lambda$ is absolutely irreducible
  because induction from $\mod {{}_f\CH_\alpha}$ to ${}_f\CO_\alpha$ is an
  equivalence of categories by \cref{thm:eqv}.

  Now suppose $\alpha=(m_1, \dots, m_b)\in \compnb nb$ and that for all
  $\lambda=(\lambda^1, \dots, \lambda^b)\in \bpart bn$ with $|\lambda|=
  \alpha$ and all $j\in [b]$, the ${}_f\CI^{m_j}$-module ${}_fP^{\lambda^j}$
  is absolutely irreducible. Then for all $j\in [b]$, ${}_f\CI^{m_j}$ is a
  split semisimple $k$-algebra with simple modules $\{\, {}_fP^{\hat
    \lambda}\mid \hat \lambda\in \CP(m_j)\,\}$. Thus, it follows from
  \cref{lem:iiso} that ${}_f\CH_\alpha$ is split semisimple and that the set
  $\{\, {}_fP^{\lambda}_0 \mid \lambda \in \bpart bn,\ |\lambda|=\alpha
  \,\}$ is a complete set of inequivalent, irreducible
  ${}_f\CH_\alpha$-modules. Therefore, by \cref{thm:eqv}, the set $\{\,
  {}_fV^{\lambda} \mid \lambda \in \bpart bn,\ |\lambda|=\alpha\,\}$ is a
  complete set of inequivalent, irreducible ${}_f\CH$-modules in
  ${}_f\CO_\alpha$. This proves the second statement.

  Finally, suppose that for all $\lambda=(\lambda^1, \dots, \lambda^b)\in
  \bpart bn$ and all $j\in [b]$, the ${}_f\CI^{m_j}$-module
  ${}_fP^{\lambda^j}$ is absolutely irreducible. Then it follows from the
  block decomposition \cref{ssec:bk}\cref{eq:bk} and what has already been
  proved that the set $\{\, {}_fV^{\lambda} \mid \lambda \in \bpart bn \,\}$
  is a complete set of inequivalent, irreducible ${}_f\CH$-modules, each of
  which is absolutely irreducible. Therefore
  \[
  \dim ({}_f\CH /\operatorname{rad} {}_f\CH) = \sum_{\lambda\in \bpart bn}
  (\dim {}_fV^\lambda)^2= \sum_{\lambda\in \bpart bn} |\SYT^\lambda|^2.
  \]
  But $\sum_{\lambda\in \bpart bn} |\SYT^\lambda|^2= b^n n!= \dim {}_f\CH$,
  and so $\operatorname{rad} {}_f\CH =0$, which implies that ${}_f\CH$ is a
  split semisimple $k$-algebra.
\end{proof}

Recall that $K$ is the quotient field of $A$.  Hoefsmit
\cite{hoefsmit:representations} has shown that for $m>1$ and $\hat\lambda
\in \CP(m)$, the ${}_K\CI^{m}$-module ${}_KP^{\hat \lambda}$ is irreducible,
and so the next corollary follows from \cref{cor:sp}\cref{it:3}.

\begin{corollary}\label{cor:K}
  The $K$-algebra ${}_K\CH$ is split semisimple and $\{\, {}_KV^\lambda\mid
  \lambda\in \bpart bn\,\}$ is a complete set of inequivalent, irreducible
  ${}_K\CH$-modules.
\end{corollary}

\section{Connections with Yokonuma-Hecke algebras}\label{sec:last}

In this section we explain the connections between the constructions and
results in this paper and the related constructions and results of Thiem
\cite{thiem:thesis}, Juyumaya \cite{juyumaya:representation}
\cite{juyumaya:nouveaux}, Chlouveraki and Poulain d'Andecy
\cite{chlouverakipoulaindandecy:representation}, and Jacon and Poulain
d'Andecy \cite{jaconpoulaindandecy:isomorphism}.

Suppose $q$ is a prime power and $f\colon \BBZ[\bq, \ba]\to \BBC$ is the
specialization with $f(\bq)=q$ and $f(\ba)=1$. Then as recalled in
\cref{ssec:sp}, the specialized algebra ${}_f \CH_{q-1,n}$ is isomorphic to
the Hecke algebra $\CH_{\BBC}(G,U)$.

\subsection{Thiem's thesis}\label{ssec:Th}
This paper is an extension of Thiem's study of $\CH_{\BBC}(G,U)$ in
\cite[Chapter 6]{thiem:thesis} to the more general case of Yokonuma-type
Hecke algebras. The main differences are the following.

\begin{enumerate}
\item This paper deals with a larger class of algebras that are ``generic''
  algebras with minimal assumptions on the underlying ring of scalars. The
  results in \cite[Chapter 6]{thiem:thesis} are obtained from the results in
  this paper by taking $b=q-1$ and setting $\bq=q$ and $\ba=1$.
\item The formulation of the block/weight space decomposition of $\CH$ is
  formalized and placed in a categorical framework.
\item More details of the structure of the representations $V^\lambda$ in
  \cref{sec:module} are worked out.
\end{enumerate}

\subsection{Juyumaya's presentation}\label{ssec:Ju}
Juyumaya \cite{juyumaya:nouveaux} found an intriguing presentation of the
Hecke algebra $\CH_{\BBC}(G,U)$ that leads to a simpler presentation for
one-parameter Yokonuma-Hecke algebras than the presentation arising from
\cref{ssec:Hbn} when $b=q-1$. Using the notation above, Juyumaya's
presentation is described as follows.

Suppose for the moment that $q$ is a power of an odd prime $p$ and $\xi$ is
a primitive $p\th$ root of unity in $\BBC$. Let $h\colon \BBZ[\bq, \ba]\to
\BBZ[\xi][\bq, \bq\inverse]$ be the specialization with $h(\bq)=\bq$ and
$h(\ba)=1$ and consider the specialized algebra ${}_h\CH_{q-1,n}$. To
streamline the notation, set $\Atilde= \BBZ[\xi][\bq, \bq\inverse]$, $\CHt=
{}_h\CH_{q-1,n}$, and denote the generators $1\otimes T_j$ and $1\otimes
R_i$ of $\CHt$ simply by $T_j$ and $R_i$, respectively. With these
conventions, the quadratic relation \cref{rel:6} becomes
\begin{equation*}
  R_{i}^2=\bq+ T_{i}^{(q-1)/2} E_i R_i \quad \text{for $i\in [n-1]$.} 
\end{equation*}

Fix a generator $\omega$ of $\BBF_q^\times$ and a non-trivial character
$\psi\colon \BBF_q\to \Atilde^\times$ of the additive group of $\BBF_q$. For
$i\in [n-1]$ define
\[
F_i= \sum_{j=1}^{q-1} \psi(\omega^j) T_{i}^j,\quad\text{and}\quad R_i'=
T_{i}^{(q-1)/2} R_{i} \quad\text{in}\quad \CHt.
\]
Juyumaya's presentation is in terms of the elements
\[
J_i= q\inverse( E_i+R_i'F_i) \quad\text{and}\quad P_i= (q^2-1)\inverse(
E_i+R_i' E_i) .
\]
The arguments in \cite[\S2]{juyumaya:nouveaux} show that $\CHt$ has
generators
\[
\text{$T_1$, \dots, $T_n$, $J_1$, \dots $J_{n-1}$}
\]
and relations \cref{rel:1}, \cref{rel:2}, and 
\begin{alignat}{2}
  &T_jJ_{i}=J_{i}T_{s_i(j)} \qquad&& \text{for $j\in[n]$ and
    $i\in [n-1]$,} \tag{$\textrm {jr}_3$} \label{rel:j3} \\
  &J_{i}J_{i'}=J_{i'}J_{i} \qquad&& \text{for $i,i'\in [n-1]$ with
    $|i-i'|>1$,} \tag{$\textrm {jr}_4$} \label{rel:j4}\\
  &J_{i}J_{i+1}J_{i}= J_{i+1}J_{i}J_{i+1} \qquad&& \text{for $i\in [n-2]$,
    and} \tag{$\textrm {jr}_5$} \label{rel:j5}\\
  &J_{i}^2= 1+ (\bq^{-2}-1)P_i \qquad&& \text{for $i\in [n-1]$.}
  \tag{$\textrm {jr}_6$} \label{rel:j6}
\end{alignat}
This is essentially the presentation in \cite[Theorem
3]{juyumaya:representation} and \cite[Theorem 2.18]{juyumaya:nouveaux}.

To get a presentation similar to that in \cref{ssec:Hbn}, for $i\in [n-1]$
define
\[
G_i=-\bq\inverse R_i'(F_i-(\bv+1)\inverse E_i),
\]
where $\bv$ is a square root of $\bq$. It can be shown that $\CHt$ has a
presentation with generators
\[
\text{$T_1$, \dots, $T_n$, $G_1$, \dots $G_{n-1}$}
\]
and relations \cref{rel:1}, \cref{rel:2}, and
\begin{alignat}{2}
  &T_jG_{i}=G_{i}T_{s_i(j)} \qquad&& \text{for $j\in [n]$ and
    $i\in [n-1]$,} \tag{$\textrm {gr}_3$} \label{rel:g3} \\
  &G_{i}G_{i'}=G_{i'}G_{i} \qquad&& \text{for $i,i'\in[n-1]$ with
    $|i-i'|>1$,} \tag{$\textrm {gr}_4$} \label{rel:g4}\\
  &G_{i}G_{i+1}G_{i}= G_{i+1}G_{i}G_{i+1} \qquad&& \text{for $i\in [n-2]$,
    and} \tag{$\textrm {gr}_5$} \label{rel:g5}\\
  &G_{i}^2= \bq+ E_iG_i \qquad&& \text{for $i\in [n-1]$.}  \tag{$\textrm
    {gr}_6$} \label{rel:g6}
\end{alignat}

Notice that if $q$ is a power of two, then relation \cref{rel:6} in
${}_f\CH_{q-1,n}$ is simply
\[
R_{i}^2= q+ E_iR_i \qquad \text{for $i\in [n-1]$,}
\]
and so the presentation just given for the one-parameter Yokonuma-Hecke
algebra $\CHt ={}_h\CH_{q-1,n}$ when $q$ is odd is valid for all prime
powers.

Because $G_i$ is defined in terms of $F_i$ and $F_i$ depends on the additive
character $\psi$, and thus on the arithmetic of the field $\BBF_q$, it is
not immediately clear how to modify Juyumaya's construction to obtain a
presentation of the Yokonuma-type Hecke algebra $\CH_{b,n}$ when $b\ne q-1$.

\subsection{Work of Chlouveraki, Jacon, and Poulain
  d'Andecy}\label{ssec:CLP}
Suppose as above that $\bv$ is a square root of $\bq$. Jacon and Poulain
d'Andecy \cite{jaconpoulaindandecy:isomorphism} define two-parameter
Yokonuma-Hecke algebras $Y_{b,n}$ that are $\BBC[\bv, \bv\inverse,
\ba]$-algebras with generators
\[
\text{$T_1$, \dots, $T_n$, $G_1$, \dots $G_{n-1}$}
\]
and relations \cref{rel:1}, \cref{rel:2}, \cref{rel:g3}, \cref{rel:g4},
\cref{rel:g5}, and the quadratic relation
\begin{alignat}{2}
  &G_{i}^2= \bv^2+ \ba (b\inverse E_i)G_i \qquad&& \text{for $i\in[n-1]$.}
  \tag{$\textrm {gr}_6'$} \label{rel:g6''}
\end{alignat}

\begin{lemma}\label{lem:spa}
  Consider the specialization $g\colon \BBZ[\bq, \ba]\to \BBC[\bv,
  \bv\inverse, \ba]$ with $g(\bq)= \bv^2$ and $g(\ba)= b\inverse \ba$. Then
  ${}_{g} \CH_{b,n} =Y_{b,n}$.
\end{lemma}

\begin{proof}
  The result is clear when $b$ is odd. 

  Suppose now that $b$ is even and for $i\in [n-1]$ define
  \[
  u_i= \sum_{\substack{\chi\in X(D) \\ s_i\cdot \chi =\chi}} T_i^{b/2}
  e_\chi + \sum_{\substack{\chi\in X(D) \\ s_i\cdot \chi \ne \chi}} e_\chi
  \qquad \text{and}\qquad \Rtilde_i = u_i R_i.
  \]
  Using the fact that
  \[
  \sum_{j=0}^{b-1} T_i^j T_{i+1}^{-j} e_\chi=
  \begin{cases}
    be_\chi& \text{if $s_i\cdot \chi =\chi$} \\
    0& \text{if $s_i\cdot \chi \ne\chi$}
  \end{cases}
  \]
  it is straightforward to check that $u_i^2=1$, $u_iR_i= R_iu_i$, and
  $T_i^{(b^2-b)/2} E_iu_i=E_i$, and hence that $\CH_{b,n}$ has generators
  \[
  \text{$T_1$, \dots, $T_n$, $\Rtilde_1$, \dots $\Rtilde_{n-1}$}
  \]
  and relations \cref{rel:1}, \cref{rel:2}, \cref{rel:g3}, \cref{rel:g4},
  \cref{rel:g5}, and the quadratic relation
  \begin{alignat}{2}
    &\Rtilde_{i}^2= \bq+ \ba E_i\Rtilde_i \qquad&& \text{for $i\in[n-1]$.}
    \tag{$\textrm {gr}_{6t}$} \label{rel:g6t}
  \end{alignat}
  It follows that ${}_{g} \CH_{b,n}= Y_{b,n}$ in this case as well.
\end{proof}

Chlouveraki and Poulain d'Andecy
\cite{chlouverakipoulaindandecy:representation} define one-parameter
Yokonuma-Hecke algebras that we denote by $Y_{b,n}'$. These are $\BBC[\bv,
\bv\inverse]$-algebras with generators
\[
\text{$T_1$, \dots, $T_n$, $G_1$, \dots $G_{n-1}$}
\]
and relations \cref{rel:1}, \cref{rel:2}, \cref{rel:g3}, \cref{rel:g4},
\cref{rel:g5}, and the quadratic relation
\begin{alignat}{2}
  &G_{i}^2= \bv^2+ (\bv^2-1) (b\inverse E_i)G_i \qquad&& \text{for $i\in
    [n-1]$.} \tag{$\textrm {gr}_6''$} \label{rel:g6'}
\end{alignat}
Setting $\Gtilde_i= \bv\inverse G_i$ for $i\in [n-1]$ gives the presentation
in \cite[2.1]{chlouverakipoulaindandecy:representation} with quadratic
relation
\[
\Gtilde_{i}^2= 1+ (\bv-\bv\inverse) (b\inverse E_i)\Gtilde_i.
\]
It follows from \cref{lem:spa} that if $g'\colon \BBZ[\bq, \ba]\to \BBC[\bv,
\bv\inverse]$ is the specialization with $g'(\bq)=\bv^2$ and
$g'(\ba)=b\inverse( \bv^2-1)$, then $Y_{b,n}'= {}_{g'}\CH_{b,n}$. If $q$ is
a prime power and $f'\colon \BBC[\bv, \bv\inverse] \to \BBC$ with
$f'(\bv)=\sqrt q$, then the quadratic relation \cref{rel:g6'} simplifies to
the relation \cref{rel:g6} and so ${}_{f'} Y_{q-1,n}' = {}_f \CH_{q-1,n}
\cong \CH_{\BBC}(G,U)$.

A family of $Y_{b,n}'$-modules indexed by $\bpart bn$ is constructed in
\cite[\S5.1]{chlouverakipoulaindandecy:representation} by defining an action
of the generators $T_j$ and $G_i$ on basis vectors and checking that the
defining relations of $Y_{b,n}'$ hold. For $\lambda\in \bpart bn$, denote
the module constructed in \cite[\S5.1]
{chlouverakipoulaindandecy:representation} by $\Vbar^\lambda$. With the
notation in \cite[\S5.1] {chlouverakipoulaindandecy:representation}, $\{\,
\bv_\CT\mid \CT\in \SYT^\lambda\,\}$ is a basis of $\Vbar^\lambda$. It is
straightforward to check that the obvious bijection between basis elements
between $\Vbar^\lambda$ and ${}_{g'} V^\lambda$, namely $\bv_{\CT}
\leftrightarrow 1\otimes v_\tau$, defines an isomorphism of $Y_{b,n}'
={}_{g'}\CH_{b,n}$-modules $\Vbar^\lambda \cong {}_{g'} V^\lambda$. Thus it
follows from \cref{thm:irrep} that each $\Vbar^\lambda$ is an induced
module.


\subsection*{Acknowledgments} The authors thank Nat Thiem and Lo\"ic Poulain
d'Andecy for very helpful communications.


\bibliographystyle{plain}

\end{document}